\documentclass[11pt,english]{smfart}
\usepackage[english,francais]{babel} 
\usepackage{amssymb,url,xspace,smfthm}
\usepackage{amsfonts,amsthm,cite,amsmath,amstext,amsbsy,bm}
\usepackage{paralist}
\usepackage{hyperref}
\usepackage{color}
\usepackage[dvipsnames]{xcolor}
\usepackage{mathrsfs}
\usepackage{mathtools}
\usepackage{euscript}
\usepackage[all]{xy}
\usepackage{dsfont}
\usepackage{graphicx}
 \usepackage{tikz-cd}
\usepackage{lipsum}

\usepackage[a4paper,text={150mm,240mm},centering]{geometry}

\newcommand{\newabstract}[1]{%
	\par\bigskip
	\csname otherlanguage*\endcsname{#1}%
	\csname captions#1\endcsname
	\item[\hskip\labelsep\scshape\abstractname.]
}

\DeclarePairedDelimiter\floor{\lfloor}{\rfloor}

\makeatletter
\newsavebox{\@brx}
\newcommand{\llangle}[1][]{\savebox{\@brx}{\(\m@th{#1\langle}\)}%
  \mathopen{\copy\@brx\kern-0.5\wd\@brx\usebox{\@brx}}}
\newcommand{\rrangle}[1][]{\savebox{\@brx}{\(\m@th{#1\rangle}\)}%
  \mathclose{\copy\@brx\kern-0.5\wd\@brx\usebox{\@brx}}}
\makeatother

\usepackage{thmtools}
\usepackage{enumitem}
\usepackage{letltxmacro}

\usepackage{nameref}
\usepackage{cleveref}

\declaretheorem[
name=Theorem,
Refname={Theorem,Theorems},
numberwithin=section]{thm}

\declaretheorem[
name=Claim,
Refname={Claim,Claims},
sibling=thm]{claim}

\declaretheorem[
name=Lemma,
Refname={Lemma,Lemmas},
sibling=thm]{lem}

\declaretheorem[
name=Definition,
Refname={Definition,Definitions},
sibling=thm]{dfn}

\newtheorem{thmx}{Theorem}
\renewcommand{\thethmx}{\Alph{thmx}}

\newlist{thmenum}{enumerate}{1} 
\setlist[thmenum]{wide = 0pt, labelwidth = 2em, labelsep*=0em, itemindent = 0pt, leftmargin = \dimexpr\labelwidth + \labelsep\relax, noitemsep,topsep = 1ex, font=\normalfont, label=(\roman*), ref=\thethmx.(\roman{thmenumi})}
\crefalias{thmenumi}{thmx}

\theoremstyle{plain}
\newlist{thmlist}{enumerate}{1}
\setlist[thmlist]{wide = 0pt, labelwidth = 2em, labelsep*=0em, itemindent = 0pt, leftmargin = \dimexpr\labelwidth + \labelsep\relax, noitemsep,topsep = 1ex, font=\normalfont, label=(\roman*), ref=\thethm.(\roman{thmlisti})}

\addtotheorempostheadhook[thm]{\crefalias{thmlisti}{thm}}

\addtotheorempostheadhook[proposition]{\crefalias{thmlisti}{proposition}}

\addtotheorempostheadhook[dfn]{\crefalias{thmlisti}{dfn}}

\addtotheorempostheadhook[lem]{\crefalias{thmlisti}{lem}}
\addtotheorempostheadhook[main]{\crefalias{thmlisti}{main}}

\crefname{lem}{Lemma}{Lemmas}
\crefname{thm}{Theorem}{Theorems}
\crefname{proposition}{Proposition}{Propositions}
\crefname{dfn}{Definition}{Definitions}
\crefname{rem}{Remark}{Remarks}
\crefname{cor}{Corollary}{Corollaries}
\crefname{corx}{Corollary}{Corollaries}
\crefname{problem}{Problem}{Problems}
\crefname{thmx}{Theorem}{Theorems}
\crefname{claim}{Claim}{Claims}

\crefname{main}{Main Theorem}{Main Theorems}

\newtheorem{rem}[thm]{Remark}
\newtheorem{conjecture}[thm]{Conjecture}

\newtheorem{corx}[thmx]{Corollary}

\makeatletter     
\newcommand{\crefnames}[3]{%
	\@for\next:=#1\do{%
		\expandafter\crefname\expandafter{\next}{#2}{#3}%
	}%
}
\makeatother

\crefnames{part,chapter,section}{\S}{\S\S}

\newcommand{\lowerromannumeral}[1]{\romannumeral#1\relax}

\numberwithin{equation}{subsection}

\def\oc{\mathscr{O}}  
 \def\rc{\mathbb{R}} \def\mc{\mathfrak{m}}
 \def\fc{\mathcal{F}}

\def\af{\mathbf{a}}

\def\cb{\mathbb{C}}

\def\ab{\mathbb{A}}

\def\Im{\operatorname{Im}}

\def\as{\mathscr{A}} \def\es{e^\star}
\def\ls{\mathscr{L}}
\def\tl{\widetilde}  

\def\cb{\mathbb{C}} 
\def\gb{\mathbb{G}} 
\def\pb{\mathbb{P}}  
\def\zbb{\mathbb{Z}}
\def\ys{\mathscr{Y}}
\def\ls{\mathscr{L}}
\def\hs{\mathscr{H}}
\def\xs{\mathscr{X}}
\def\zs{\mathscr{Z}}
\def\grs{{\rm{Gr}}_{k+1}(V)}
\def\ebf{\mathbf{e}}

\def\gf{\mathbf{G}}

\def\wk{\mathfrak{w}}
\def\pt{\text{P}}

 \def\d{\partial}

\def\ep{\varepsilon}
\def\es{\mathscr{E}}

\let\oldsection\section
\renewcommand{\section}{
	\renewcommand{\theequation}{\thesection.\arabic{equation}}
	\oldsection}
\let\oldsubsection\subsection
\renewcommand{\subsection}{
	\renewcommand{\theequation}{\thesubsection.\arabic{equation}}
	\oldsubsection}

\title{On the Diverio-Trapani Conjecture}
\alttitle{Autour de la conjecture de Diverio-Trapani}
\author{Ya Deng}
\address{Institut Fourier, Universit\'e Grenoble Alpes, 100 Rue des Maths, Gi\`eres, 38610, France}
\curraddr{Universit\'e de Strasbourg, Institut de Recherche Math\'ematique Avanc\'ee,   	7 Rue Ren\'e-Descartes, Strasbourg, 67084, France}
\email{deng@math.unistra.fr;\ dengya.math@gmail.com}
\urladdr{https://sites.google.com/site/dengyamath}

\begin{document}

	\begin{abstract}
	In this paper we establish   effective lower bounds on the degrees of the Debarre and Kobayashi conjectures.   Then we study a more general conjecture  proposed  by Diverio-Trapani on the ampleness of  jet bundles of general complete intersections in complex projective spaces.
\end{abstract}
\begin{altabstract}
	Dans cet article,  nous \'etablissons  des bornes  inf\'erieures effectives sur les  degr\'es li\'es aux conjectures de Debarre et Kobayashi.  Ensuite,   nous \'etudions une conjecture plus g\'en\'erale propos\'ee par Diverio-Trapani sur l'amplitude des fibr\'es de jets des intersections compl\`etes g\'en\'erales  dans les espaces projectifs complexes.  
	\end{altabstract}
 \altkeywords{hyperbolicit\'e au sens de Kobayashi, fibr\'e cotangent ample,  conjecture de Debarre,  conjecture de  Kobayashi,  conjecture de Diverio-Trapani,  th\'eor\`eme de Nakamaye,  Wronskians de Brotbek}
	\maketitle



\section{ Introduction}\label{intro}
A compact complex manifold $X$ is said to be Kobayashi (Brody) hyperbolic if there exists no non-constant  holomorphic map $f:\cb\to X$.  As is well-known, a sufficient criteria for Kobayashi hyperbolicity is the ampleness of the cotangent bundle.  Although the complex manifolds  with ample cotangent bundles   are expected to be reasonably abundant, there are few concrete constructions before the work of Debarre.  In \cite{Deb05}, Debarre  proved that the complete intersection of sufficiently ample general  hypersurfaces in a \emph{complex abelian variety},  whose codimension is at least as large as its dimension, has ample cotangent
bundle.  He further conjectured that this result should also hold for intersection varieties of general hypersurfaces in \emph{complex projective spaces} (the so-called \emph{Debarre conjecture}). This conjecture was recently proved by Brotbek-Darondeau \cite{BD15} and independently  by Xie \cite{Xie16,Xie15},  based on the ideas and explicit methods  in \cite{Bro16}. 
\begin{thm}[Brotbek-Darondeau, Xie]\label{bdx}
Let  $X$ be an $n$-dimensional  projective manifold  equipped with a very ample line 	bundle $\as$. Then   there exists   $d_{{\rm Deb},n}\in \mathbb{N}$ depending only on the dimension $n$, such that for all $d \geqslant d_{{\rm Deb},n}$,  the   complete intersection of $c$-general  hypersurfaces   $H_1,\ldots,H_c\in |\as^d|$  has ample cotangent bundle,  provided that $\frac{n}{2}\leqslant c\leqslant n$.
\end{thm}
In \cite{Xie15},  Xie was able to obtain  an effective  lower bound   $d_{{\rm Deb},n}=  n^{n^2}$ by working with (much more elaborated) explicit expressions of some symmetric differential forms.  The result  in \cite{BD15}    is \og almost\fg effective on $d_{{\rm Deb},n}$, because it depends on some constant involved in some noetherianity argument, arising in their reduction to Nakamaye's theorem \cite{Nak00} for \emph{families of  zero-dimensional subschemes}.  
\medskip

One goal of the present paper  is to provide an effective estimate for such a Nakamaye's theorem (see \cref{effective nakamaye2}). In particular,  as a complement of   \cite[Theorem 1.1]{BD15},  we can improve Xie's effective lower bound   $d_{{\rm Deb},n}$.
\begin{thmx}\label{main:Debarre}
	In the same setting  as \cref{bdx}, one can take
	$$d_{{\rm Deb},n}=   (2n)^{n+3}.$$ 
\end{thmx}

It is worth to  mention that the techniques in \cite{BD15} are more intrinsic and  the ideas of their proof brought  new geometric insights  in the understanding of the positivity of cotangent bundles.  Later,  Brotbek  \cite{Bro17} extended these techniques from the setting of symmetric differentials to that of higher order jet differentials, so that he  was able to prove a long-standing conjecture of  Kobayashi in \cite{Kob70}. 
\begin{thm}[Brotbek]\label{Brotbek}
	Let $X$ be a projective manifold of dimension $n$.  For any very ample line bundle $\as$ on $X$, there exists $d_{{\rm Kob},n}\in \mathbb{N}$ depending only on the dimension $n$  such that for  any $d\geqslant d_{{\rm Kob},n}$,   a general smooth hypersurface    $H\in |\as^d|$ is Kobayashi hyperbolic.  
\end{thm}
The proof of \cref{Brotbek} in \cite{Bro17}  is also \og almost\fg effective on $d_{{\rm Kob},n}$  because of two noetherianity arguments:  the first concerns the increasing  sequences of Wronskians ideal sheaves;  the second  concerns a constant arising in   Nakamaye's theorem as that of \cite{BD15},  which can be made effective by \cref{weak nakamaye}.  Our second goal  of the present  paper is  to give an intrinsic interpretation of Brotbek's Wronskians (see \cref{brotbek wronskian}), and as a byproduct, we can   render the above-mentioned first  noetherianity argument effective. This in turn provides  effective lower bounds  for the Kobayashi conjecture in combination with the explicit formula of $d_{{\rm Kob},n}$ in \cite{Bro17}.
 \begin{thmx}\label{main:Kobayashi}
 	In the same setting as \cref{Brotbek},  one can take
 	$$d_{{\rm Kob},n}=  n^{2n+3}(n+1).$$
 \end{thmx}
Let us mention that  in \cite{Bro17} Brotbek obtained a much stronger result than \cref{Brotbek}. Indeed, he proved that for the hypersurface $H$ in \cref{Brotbek}, the tautological line bundle $\oc_{H_k}(a_k,\ldots,a_1)$ on the \emph{Demailly-Semple $k$-jet tower} $H_k$ of the direct manifold $(H,T_H)$ is \og almost ample\fg for some $(a_1,\ldots,a_k)\in \mathbb{N}^k$ when $k\geqslant n-1= {\rm dim}\, H$. In view of the following vanishing theorem by Diverio in \cite{Div08}, the above-mentioned lower bound for $k$ in \cite{Bro17}  is optimal.
\begin{thm}[Diverio]
	Let $Z\subset \pb^n$ be a smooth complete intersection of  hypersurfaces of any degree in $\pb^n$. Then
	$$
	H^0(Z,E_{k,m}^{\rm GG}T_Z^*)=0
	$$
	for all $m\geqslant1$ and $1\leqslant k<{\rm dim}(Z)/{\rm codim}(Z)$. Here $E_{k,m}^{\rm GG}T_Z^*$ denotes the \emph{Green-Griffiths jet bundle} of order $k$ and weighted degree $m$.
\end{thm}

Motivated by the above vanishing theorem, in the same vein as the Debarre conjecture,  Diverio-Trapani  proposed  the following generalized conjecture in \cite{DT10}.\begin{conjecture}[Diverio-Trapani]\label{DT}
	Let $Z\subset \pb^n$
	be the complete intersection of $c$-general hypersurfaces of sufficiently high degree. Then the \emph{invariant jet bundle} $E_{k,m}T_Z^*$ is
	ample provided that $k\geqslant \frac{n}{c}-1$ and $m\gg 0$.   
\end{conjecture}

The last aim of the present paper is to study \Cref{DT} using geometric methods in \cite{BD15,Bro17}.  
\begin{thmx}\label{main}
	Let $X$ be an $n$-dimensional projective manifold  equipped with a very ample line bundle $\as$, and let $Z\subset X$
	be the complete intersection of $c$-general hypersurfaces  $H_1 \ldots,H_c\in \vert \as^{d} \lvert$. Then  $Z$ is \emph{almost $k$-jet ample}  (see  \cref{almost k ample}) if   $k\geqslant \frac{n}{c}-1$, and     $d \geqslant 2c n^{c\lceil \frac{n}{c}\rceil+1}\cdot \lceil \frac{n}{c}\rceil^{c\lceil \frac{n}{c}\rceil+3} $.
  In particular, $Z$ is Kobayashi hyperbolic.
\end{thmx} 
Let us mention that we apply the results in the first part of the present paper to obtain the effective   lower degree bounds in \cref{main}.  

In view of the correspondence between tautological line bundles on the Demailly-Semple jet towers and invariant jet bundles studied  in \cite[Proposition 6.16]{Dem95},  the following result on \Cref{DT}  is a consequence of \cref{main}.
\begin{corx}\label{main 3}
In the same setting as \cref{main},  for any $k\geqslant \frac{n}{c}-1$, 
there exists a subbundle $\mathscr{F}\subset E_{k,m}T^*_Z$ for some $m\gg 0$ such that
\begin{thmlist}
	\item $\mathscr{F}$ is ample.
	\item For any regular  germ of curve $f:(\cb,0)\to (Z,z)$, there is  a global section  $P\in H^0(Z, \mathscr{F}\otimes \as^{-1})$ so that $P([f]_k)(0)\neq 0$.  
	\end{thmlist}
\end{corx}
In other words, one  can find a subbundle $\mathscr{F}$ of the invariant jet bundle $E_{k,m}T^*_Z$, which is ample, and the \emph{Demailly-Semple locus} (see \cite[\S  2.1]{DR15} for the definition)  induced by $\mathscr{F}$   is empty.
\medskip

Lastly, let us mention that the techniques in \cite{BD15,Bro17} were extended by Brotbek and the author to prove a logarithmic analogue of the Debarre conjecture  in \cite{BD17},  and to prove the logarithmic (orbifold) Kobayashi conjecture in \cite{BD18}.  To achieve the effective  lower degree bounds, both the articles \cite{BD17,BD18}  rely on  the methods in the present paper.
 \medskip

This paper is organized as follows. In \cref{sec:jet spaces}  we recall the fundamental tools of jet differentials by Demailly, Green-Griffiths and Siu, which can be seen as higher order analogues of symmetric differential forms and provide obstructions to the existence of entire curves.    \cref{brotbek wronskian} is devoted to the  study of  new  techniques of  Wronskians introduced by Brotbek  in his proof of  the Kobayashi conjecture \cite{Bro17}.  
We bring a new perspective of Brotbek's Wronskians, which we interpret as  a certain \emph{morphism of $\oc$-modules} from   the jet bundles of   a line bundle to the invariant jet bundles.  
In view of this result one can immediately make  the first noetherianity argument in \cite{Bro17} effective. In   \cref{Nakamaye}, by means of an explicit construction of global sections with a \og negative twist\fg, we obtain a \emph{slightly weaker but effective} Nakamaye's theorem for the universal families of zero-dimensional subschemes introduced in \cite{BD15,Bro17}. This in turn renders the second noetherianity argument in \cite{Bro17}  as well as that in \cite{BD15}  effective, and in combination with the formulas for lower degree bounds in \cite{BD15,Bro17}, we prove \cref{main:Debarre,main:Kobayashi}.   The  aim of \cref{proof of main}  is to study \Cref{DT}. In \cref{sec:Fermat} we briefly recall  the essential results in \cite{Bro17}, and we show in  \cref{sec:intersection} and \cref{main proof} how to deduce \cref{main} from Brotbek's techniques.

\medskip

\noindent
\textbf{Acknowledgements.}
I would like to warmly thank my thesis supervisor Professor Jean-Pierre Demailly for his
constant encouragements and supports, and  Damian Brotbek for suggesting this problem and kindly sharing his ideas to me. I also thank Professors Steven Lu and Erwan Rousseau for their interests and suggestions on the work. I am  indebted to  Lionel Darondeau and Songyan Xie for their discussions.  Lastly,  I thank the anonymous referee for very helpful  suggestions to improve the presentation in this paper.  This work is supported by the  ERC ALKAGE project.

\section{Jet differentials and Brotbek's Wronskians}\label{technical} 
 By the work of Nadel \cite{Nad89} and Demailly-El Goul \cite{DE97},  the \emph{Wronskians} induced by meromorphic connections  provide an  abundant supply of invariant jet differentials. In \cite{Bro17} Brotbek introduced an alternative approach to construct Wronskian jet differentials associated to sections of a given line bundle. In \cref{brotbek wronskian}  we give an intrinsic definition of  Brotbek's Wronskians via the jet bundles of   line bundles.
\subsection{Jet spaces and jet differentials}\label{sec:jet spaces}
In this subsection, we collect the main techniques of jet differentials in \cite{Dem95}. A \emph{direct manifold} is a pair $(X, V)$  where $X$
is a complex manifold and $V\subset T_X$  is a holomorphic subbundle of the tangent bundle.  Denote by $p_k:J_kV\rightarrow X$   the bundle of $k$-jets of germs of parametrized curves in $(X,V)$, that is, the set of equivalent
classes of holomorphic maps $f:(\cb, 0)\rightarrow (X,x)$ which are tangent to $V$, with the equivalence relation $f\sim g$ if and only if all
derivatives $f^{(j)}(0)= g
^{(j)}(0)$ coincide for $0\leqslant j\leqslant k$, when computed in some local coordinate system
of $X$ near $x$. The class $f$ in $J_kV$ is denoted by $[f]_k$.  
The projection map  $p_k:J_kV\rightarrow X$ is simply  $[f]_k\mapsto f(0)$. When $V=T_X$, we simply write $J_kX$ in place of $J_kV$. Note that $J_kX\to X$ is  a local trivial fibration with fibers $\cb^{nk}$. Indeed, local coordinates  $(z_1,\ldots,z_n)$  for  an open set $U\subset X$ induce coordinates
\[
(z_1,\ldots,z_n,z_1',\ldots,z_n',\ldots,z_1^{(k)},\ldots,z_n^{(k)})
\]
on $p_k^{-1}(U)$, and any $k$-jet  $[f]_k\in p_k^{-1}(U)$ has coordinates    $$\big(f_1(0),\ldots,f_n(0),\ldots,f_1^{(k)}(0),\ldots,f_n^{(k)}(0)\big).$$

 Let $\gb_k$ be the group of germs of $k$-jets of biholomorphisms of $(\cb, 0)$, that is, the group of germs of biholomorphic maps
$$t \mapsto \varphi(t) = a_1t + a_2t^2+ \cdots + a_kt^k,\quad a_1\in \cb^*, a_j \in \cb, \forall j\geqslant 2,$$
in which the composition law is taken modulo terms $t^j$
of degree $j > k$.  Then $\gb_k$  admits a natural fiberwise right action on $J_kX$ defined by $\varphi\cdot [f]_k:=[f\circ \varphi]_k$.  Note that $\cb^*$ can be seen as a subgroup of $\gb_k$  defined by $(a_2=\cdots=a_k=0)$.  

In \cite{GG79}, Green-Griffiths introduced the vector bundle $E_{k,m}^{\rm GG}T_X^*\to X$ whose fibres are
complex valued polynomials $Q([f]_k)$  on the fibres of $J_kX$, of weighted
degree $m$ with respect to the $\cb^*$-action, that is,
 $Q(\lambda\cdot [f]_k) = \lambda^mQ(  [f]_k),$ 
for all $\lambda\in \cb^*$ and $[f]_k\in J_kX$. 
Let $U\subset X$ be an open set with local coordinates $(z_1,\ldots,z_n)$. Then any local section $Q\in  E_{k,m}^{\rm GG}T^*_{X}(U)$ can be written as
$$
Q=\sum_{|\alpha_1|+2|\alpha_2|+\cdots+k|\alpha_k|=m} c_\alpha(z) (d^{1}z)^{\alpha_1}(d^{2}z)^{\alpha_2}\cdots (d^kz)^{\alpha_k},
$$
where $c_\alpha(z)\in  \oc(U)$ for any $\alpha:=(\alpha_1,\ldots,\alpha_k)\in (\mathbb{N}^n)^{k}$, such that  for any holomorphic map  $\gamma:\Omega\rightarrow U$ from an open set $\Omega\subset \cb$, one has
$$
Q\big([\gamma]_k\big)(t)=\sum_{|\alpha_1|+2|\alpha_2|+\cdots+k|\alpha_k|=m} c_\alpha\big(\gamma(t)\big) \big(\gamma'(t)\big)^{\alpha_1}\big(\gamma''(t)\big)^{\alpha_2}\cdots \big(\gamma^{(k)}(t)\big)^{\alpha_k}\in  \oc(\Omega),
$$
where $[\gamma]_k(t):\Omega\rightarrow J_kX_{\upharpoonright U}$ is the  lifted holomorphic curve on $J_kX$ induced by $\gamma$.

The bundle
 $E_{k,\bullet}^{\rm GG}T_X^*:=\bigoplus_{m\geqslant 0}E_{k,m}^{\rm GG}T_X^*$   
is in a natural way  a bundle of graded algebras (the product is obtained simply by taking the product of polynomials).  There are natural inclusions $E_{k,\bullet}^{\rm GG}T_X^*\subset E_{k+1,\bullet}^{\rm GG}T_X^*$ of algebras, hence 
 $
E_{\infty,\bullet}^{\rm GG}T_X^*:=\bigcup_{k\geqslant 0} E_{k,\bullet}^{\rm GG}T_X^*
$ 
is also an  algebra.  It follows from \cite[\S 6]{Dem95} that the sheaf of holomorphic sections $\oc(E_{\infty,\bullet}^{\rm GG}T_X^*)$ admits a canonical 
derivation $D$
given by a collection of $\cb$-linear maps
\begin{align}\label{canonical derivative}
D:\oc(E_{k,m}^{\rm GG}T_X^*)\to \oc(E_{k+1,m+1}^{\rm GG}T_X^*)
\end{align}
constructed as follows.  For any germ of curve $f:(\cb,0)\to X$, and any $Q\in \oc(E_{k,m}^{\rm GG}T_X^*)$, 
 $$
(DQ)([f]_{k+1})(t):=\frac{d  }{dt}Q([f]_{k})(t).
$$ 
We can also inductively define $D^k:=D\circ D^{k-1}$.  In particular, for any holomorphic function $s\in \oc(U)$,   $
  D^k(s)\in   E_{k,k}^{\rm GG}T_X^*(U) 
  $.  

In this present paper, we are   interested in the more geometric context introduced by Demailly in \cite{Dem95}: the subbundle $E_{k,m}T_X^*\subset E_{k,m}^{\rm GG}T_X^*$ which consists of  polynomial differential
operators $Q$ which are invariant under arbitrary changes of parametrization, that is, for any $\varphi\in \gb_k$ and  any $[f]_k\in J_kX$, one has 
$$
Q\big(\varphi\cdot [f]_k\big)=\varphi'(0)^mQ([f]_k) .
$$
The bundle $E_{k,m}T_X^*$ is called the  \emph{invariant jet bundle  of order $k$ and weighted degree $m$}. It is noticeable that  \emph{Wronskians} provide a very natural construction for invariant jet differentials.

 For any direct manifold $(X,V)$ with ${\rm rank}\, V=r$, Demailly \cite{Dem95} introduced a \emph{fonctorial construction} of a sequence of direct manifolds
 \begin{align}\label{eq:sequence}
  \cdots\rightarrow (\pt_kV,V_k) \xrightarrow{\pi_k}(\pt_{k-1},V_{k-1})\xrightarrow{\pi_{k-1}}\cdots \xrightarrow{\pi_2} (\pt_1V,V_1)\xrightarrow{\pi_1}(\pt_0V,V_0)=(X,V)
 \end{align}
 so that $\pt_kV:=\pt(V_{k-1})$ is a $\pb^{r-1}$-bundle over $\pt_{k-1}V$ for each $k\geqslant 1$, and we say $\pt_kV$  the \emph{Demailly-Semple $k$-jet tower} of $(X,V)$. In the absolute case $(X,T_X)$, we simply write $X_k:=\pt_kV$. In the case of   smooth family of compact complex manifolds $\xs\to T$, $\xs^{\rm rel}_k$ denotes to be the Demailly-Semple $k$-jet tower  of the   direct manifold $(\xs,T_{\xs/T})$, where $T_{\xs/T}$ denotes  the relative tangent bundle. It follows from \cite[\S 6]{Dem95} that the Demailly-Semple jet tower has the following geometric properties.
 \begin{enumerate}[leftmargin=0.7cm]
 	\item Any germ of curve $f:(\cb,0)\to X$ tangent to $V$ can be lifted to $f_{[k]}:(\cb,0)\to \pt_kV$.
 	\item Denote by $J_k^{\rm reg}V:=\{[f]_k\mid f'(0)\neq 0 \}$ the set of \emph{regular $k$-jets} tangent to $V$. Then there exists a morphism $J_k^{\rm reg}V\to \pt_kV$, which sends $[f]_k$ to $f_{[k]}(0)$, whose image is a Zariski open subset $\pt_kV^{\rm reg}\subset \pt_kV$ which can be identified with the quotient $J_k^{\rm reg}V/\gb_k$. Moreover, the complement $\pt_kV^{\rm sing}:=\pt_kV\setminus \pt_kV^{\rm reg}$ is a divisor in $\pt_kV$.
 	\item For any $k,m\geqslant 0$ one has
 		\begin{align}\label{local isomorphism}
 	(\pi_{0,k})_*\oc_{\pt_kV}(m)= E_{k,m}V^*,
 	\end{align}
 	where we write $\pi_{j,k}=\pi_{j+1}\circ\cdots\circ \pi_k:\pt_kV\to \pt_jV$ for any $0\leqslant j\leqslant k$, and $\oc_{\pt_kV}(1)$ denotes the tautological line bundle over $\pt_kV=\pt(V_{k-1})$.
 \end{enumerate} 
More generally, for a $k$-tuple $(a_1,\ldots,a_k)\in \mathbb{N}^k$, we write
$$
\oc_{\pt_kV}(a_k,\ldots,a_1):=\oc_{\pt_kV}(a_k)\otimes \pi_{k-1,k}^*\oc_{\pt_{k-1}V}(a_{k-1})\otimes \cdots\otimes \pi_{1,k}^*\oc_{\pt_1V}(a_1).
$$
The \emph{fundamental vanishing theorem} shows that the jet differentials vanishing along any ample divisor gives rise to obstructions to the existence of entire curves.
\begin{thm}[Demailly, Green-Griffiths, Siu-Yeung]\label{fundamental}
	Let $(X,V)$ be any direct manifold equipped with an ample line bundle $\as$. For any non-constant entire curve $f:\cb\to X$ tangent to $V$, 
	and any $\omega\in H^0\big(\pt_kV, \oc_{\pt_kV}(a_k,\ldots,a_1)\otimes \pi_{0,k}^*\as^{-1}\big)$ with $(a_1,\ldots,a_k)\in \mathbb{N}^k$, one has
	$f_{[k]}(\cb)\subset (\omega=0)$.
\end{thm}
Observe that for any non-constant  entire curve $f:\cb \to X$ tangent to $V$, the image of its lift $f_{[k]}:\cb\to \pt_kV$  is not entirely contained in $\pt_kV^{\rm sing}$. In view of  \cref{fundamental}, we introduce the following definition.\noindent
 
\begin{dfn}\label{almost k ample}
	Let $X$ be a projective manifold. We say that $X$  \emph{is almost $k$-jet ample} if  there exists  some  $(a_1,\ldots,a_k)\in \mathbb{N}^k$ so that $\oc_{X_k}(a_k,\ldots,a_1)$ is big and its augmented base locus   ${\mathbf{B}_+}\big(\oc_{X_k}(a_k,\ldots,a_1)\big)\subset X_k^{\rm sing}.$ In particular, $X$ is Kobayashi hyperbolic. 
\end{dfn}
Note that almost 1-jet ampleness is equivalent to the ampleness of cotangent bundle.

\subsection{Brotbek's Wronskians}\label{brotbek wronskian}
This subsection is devoted to the  study of the Wronskians defined by Brotbek in \cite[\S 2.2]{Bro17}.  Let $X$ be an $n$-dimensional compact complex manifold.  Recall that for any holomorphic line bundle $L$  on $X$, one can define the bundle $J^kL$ of \emph{$k$-jet sections} of $L$ by $J^kL_x=\oc(L)_x/\big(\mc_x^{k+1}\cdot \oc(L)_x\big)$ for every $x\in X$, where $\mc_x$ is the maximal ideal of $\oc_x$.  
Pick an open set  $U\subset X$ with coordinates $(z_1,\ldots,z_n)$ so that $L_{\upharpoonright U}$ can be trivialized by a nowhere vanishing section $e_U\in L(U)$.   The fiber $ J^kL_x$ can be identified with the set of Taylor developments of order $k$
$$
\sum_{|\gamma|\leqslant k} c_\gamma (z-x)^\gamma\cdot e_U,
$$
and the coefficients $\{c_\gamma\}_{\gamma\in \mathbb{N}^n, |\gamma|\leqslant  k}$ define coordinates along the fibers of $J^kL$. This in turn gives rise to a  natural local trivialization of $J^kL$ defined by
\begin{eqnarray*}
	\Psi_U:U\times \mathbb{C}^{I_{n,k}}&\xrightarrow{\simeq}& J^kL_{\upharpoonright U},\\
	(x,c_\gamma)&\mapsto& \sum_{\gamma\in I_{n,k}} c_\gamma (z-x)^\gamma\cdot e_U,
\end{eqnarray*} 
where $I_{n,k}:=\{\gamma=(\gamma_1,\ldots,\gamma_n)\in \mathbb{N}^n\mid |\gamma|\leqslant k \}.$ 
Observe that there exists a $\cb$-linear morphism
$$
j^k_L:L\to J^kL,
$$
which is not a morphism of $\oc_X$-modules, defined as follows.  For any \(s\in L(U)\), define
\begin{align}\label{eq:jk}
j^k_{L}(s)(x):=\sum_{|\gamma|\leqslant k}\frac{1}{\gamma!}\frac{\partial^{|\gamma|}s_U}{\partial z^\gamma}(x)(z-x)^\gamma\cdot e_U,
\end{align}
where   $s_U\in \oc(U)$  so that $s=s_U\cdot e_U$.  
When $L=\oc_X$, we simply write $j^k:=j^k_{\oc_X}$. 
The jet bundle $J^kL$ will be used to interpret the canonical derivative $D:\oc(E_{k,m}^{\rm GG}T_X^*)\to \oc(E_{k+1,m+1}^{\rm GG}T_X^*)$  defined in \eqref{canonical derivative} in an alternative way.  
Let us first  give a more precise expression of $D$. 
\begin{lem}\label{d jet}
	 Take any open set $U\subset X$ with coordinates $(z_1,\ldots,z_n)$.  For any $k\geqslant1$, and any holomorphic function $s\in \oc(U)$, one has
\begin{align}\label{expression}
D^k(s)(z)=\sum_{|\alpha_1|+2|\alpha_2|+\cdots+k|\alpha_k|=k} c_{k,\alpha}(z) (d^{1}z)^{\alpha_1}(d^{2}z)^{\alpha_2}\cdots (d^kz)^{\alpha_k}\in E_{k,k}^{\rm GG}T_X^*(U)
\end{align}
such that for each $\alpha:=(\alpha_1,\ldots,\alpha_k)\in (\mathbb{N}^n)^{k}$, $c_{k,\alpha}(z)\in \oc(U)$ is a $\zbb$-linear combination of $\frac{\d^{|\gamma|} s}{\d z^{\gamma}}(z)$ with $|\gamma|=\gamma_1+\cdots+\gamma_n\leqslant k$. 
\end{lem}
\begin{proof}
We will prove the lemma by induction on $k$. For $k=1$, we simply have
$$
D(s)=ds=\sum_{i=1}^{n}\frac{\d s}{\d z_i}(z)d z_i\in T^*_X(U),
$$
and thus \eqref{expression} remains valid  for $k=1$.

Now we assume that  $D^k(s)$ has the form  \eqref{expression}. By   \eqref{canonical derivative}, one has
\begin{align*}
&D^{k+1}(s)=\\
&\sum_{|\alpha_1|+2|\alpha_2|+\cdots+k|\alpha_k|=k}\Bigg( \sum_{i=1}^{k-1}\sum_{\substack{j=1,\ldots n\\\alpha_i-\ebf_j\in \mathbb{N}^n}}c_{k,\alpha}(z) (d^{1}z)^{\alpha_1}\cdots(d^{i}z)^{\alpha_i-\ebf_j}(d^{i+1}z)^{\alpha_{i+1}+\ebf_j}\cdots (d^kz)^{\alpha_k}\\
\!\!\!\!\!\!\!\!\!\!& + \sum_{j=1}^{n}\frac{\d c_{k,\alpha}(z)}{\d z_j} (d^{1}z)^{\alpha_1+\ebf_j}\cdots (d^kz)^{\alpha_k} + \sum_{\substack{j=1,\ldots n\\\alpha_k-\ebf_j\in \mathbb{N}^n}}c_{k,\alpha}(z) (d^{1}z)^{\alpha_1}\cdots (d^kz)^{\alpha_k-\ebf_j}(d^{k+1}z)^{\ebf_j}\Bigg)
\end{align*}
where $\ebf_j$   denotes the vector in $\mathbb{N}^n$ with a $1$ in the   $j$th coordinate and 0's elsewhere. By the assumption,  for every $j=1,\ldots,n$ and every $\alpha$,   $\frac{\d c_{k,\alpha}(z)}{\d z_j}\in \oc(U)$  is a $\zbb$-linear combination of $\frac{\d^{|\gamma|} s}{\d z^{\gamma}}(z)$ with $|\gamma|=\gamma_1+\cdots+\gamma_n\leqslant k+1$.  From the above expression  we conclude that \eqref{expression} also holds true for $D^{k+1}(s)$.  The lemma follows.
\end{proof}

It  follows from \eqref{eq:jk}  and  \cref{d jet}  that  there exists a morphism   of   $\oc_X$-modules, denoted by $j^kD:J^k\oc_X\to E_{k,k}^{\rm GG}T_X^*$, so that $D^k:\oc_X\to E_{k,k}^{\rm GG}T_X^*$ factors through $j^kD $, that is, $D^k=j^kD \circ j^k$.

 Following \cite{Bro17}, given $k+1$ holomorphic functions $g_0,\ldots,g_k\in \oc(U)$, one can associate them to a    jet differentials of order $k$ and weighted degree $k':=\frac{k(k+1)}{2}$, say \emph{Wronskians}, in the following way
\begin{align}\label{def:Wronskians}
W_U(g_0,\ldots,g_k):=\begin{vmatrix}
g_0 & g_1 & \ldots & g_k \\
D(g_0) & D(g_1) & \cdots & D(g_k) \\
\vdots &\vdots & \ddots & \vdots \\
D^k(g_0) & D^k(g_1) & \cdots & D^k(g_k) 
\end{vmatrix} \in E^{\rm GG}_{k,k'}T_X^*(U).
\end{align}
It follows from \cite[Proposition 2.2]{Bro17} that   Wronskians are indeed \emph{invariant jet differentials}. 
From its alternating property, $W_U$  induces a $\cb$-linear map, which we still denoted by  $W_U:\Lambda^{k+1}\oc(U)\to E_{k,k'}T_X^*(U)$ abusively. By the factorization property of $D^k$, $W_U$ gives rise to  a morphism of $\oc_U$-module
$$
W_{J^k\oc_U}:\Lambda^{k+1}J^k\oc_U \to E_{k,k'}T_U^*
$$
so that one has
$$
W_U(g_0,\ldots,g_k)=W_{J^k\oc_U}\big(j^k(g_0)\wedge\cdots\wedge j^k(g_k)\big).
$$ 
In other words,  Brotbek's Wronskians $W_U$ can be factorized as follows.
\begin{align}\label{eq:factorization}
W_U:\Lambda^{k+1}\oc(U)\xrightarrow{\Lambda^{k+1}j^k}\Lambda^{k+1}  \big(J^k\oc_U(U)\big)\rightarrow  \big(\Lambda^{k+1}  J^k\oc_U\big)(U)\xrightarrow{W_{J^k{\oc_U}}}  E_{k,k'}T_U^*(U).
\end{align}

 Now we consider the Demailly-Semple $k$-jet tower $X_k$ of $(X,T_X)$. For the open set $U_k:=\pi^{-1}_{0,k}(U)$ of $X_k$,   the coordinate system $(z_1,\ldots,z_n)$ on $U$  induces a trivialization  $
U_k\simeq U\times \rc_{n,k},
$ 
where $\rc_{n,k}$ is some smooth  rational variety introduced in \cite[Theorem 9.1]{Dem95}. Hence
\begin{eqnarray}\label{local-pull-back}
\oc_{X_k}(1)_{\upharpoonright U_k}\simeq {\rm pr}_2^*(\oc_{\rc_{n,k}}(1)),
\end{eqnarray}
where ${\rm pr}_2:U_k\xrightarrow{\simeq} U\times \rc_{n,k}\rightarrow \rc_{n,k}$ is the composition of the isomorphism with the projection map. By   \eqref{local isomorphism}, 
we conclude that, under the above trivialization, the direct image $(\pi_{0,k})_*$ induces a local trivialization of the vector bundle $E_{k,k'}T_U^*$ 
\begin{eqnarray}\label{jet trivial}
\varphi_U: U\times H^0\big(\rc_{n,k},\oc_{\rc_{n,k}}(k')\big)\xrightarrow{\simeq} E_{k,k'}T_U^*.
\end{eqnarray}
 Write $F_{n,k}:=H^0\big(\rc_{n,k},\oc_{\rc_{n,k}}(k')\big)$.   
Therefore, under the trivializations $\varphi_U$ and $\Psi_U$, the morphism of $\oc_U$-module $W_{J^k{\oc_U}}$ is indeed  \emph{constant}, \emph{i.e.} there is a $\cb$-linear map 
${\nu}_{n,k}:\Lambda^{k+1}\cb^{I_{n,k}}\rightarrow F_{n,k} $ such that one has the following diagram.
\begin{displaymath}
\xymatrix{ U\times \Lambda^{k+1}\cb^{I_{n,k}} \ar[d]_{\Psi_U}^{\rotatebox{270}{$\simeq$}} \ar[r]^-{\mathds{1}_U\times {\nu}_{n,k}} & U\times {F_{n,k} } \ar[d]^-{\varphi_U}_{\rotatebox{90}{$\simeq$}}     \\
	\Lambda^{k+1} J^k\oc_U \ar[r]^{W_{J^k{\oc_U}}}& E_{k,k'}T_U^*   
}
\end{displaymath}
Denote by $\mathfrak{I}_{n,k}\subset \oc_{\rc_{n,k}}$ the base ideal of the linear system $ 
|{\rm Im}(\nu _{n,k})| \subset |\oc_{\rc_{n,k}}(k')|
$, and set  $\wk_{k,U}$ to be the ideal sheaf ${\rm pr}_2^*(\mathfrak{I}_{n,k})$ on $U_k$.

By \cite{Bro17}, Wronskians can also be associated to global sections of any line bundle $L$. Take  an open set  $U\subset X$ with coordinates $(z_1,\ldots,z_n)$ so that $L_{\upharpoonright U}$ can be trivialized by a nowhere vanishing section $e_U\in L(U)$.  Consider  any $s_0,\ldots,s_k\in H^0(X,L)$.  There exists  unique $s_{i,U}\in \oc(U)$ so that  $s_i=s_{i,U}\cdot e_U$ for every $i=0,\ldots,k$. It was proved in \cite[Proposition 2.3]{Bro17} that the section 
 \begin{align}\label{eq:local wronskian}
  W_{U}(s_{0,U},\ldots, s_{k,U})\cdot e_U^{k+1}\in   (E_{k,k'}T_X^*\otimes L^{k+1})(U).
 \end{align}
is \emph{intrinsically} defined, \emph{i.e.} it does not depend on the choice of $e_U$. Hence they can be glued together into a global section, denoted to be
 $
  W(s_0,\ldots,s_k)\in H^0(X,E_{k,k'}T_X^*\otimes L^{k+1}) 
  $. 
  Set
\begin{align}\label{inverse wronskian}
\omega(s_0,\ldots,s_k):=(\pi_{0,k})_*^{-1}W(s_0,\ldots,s_k)\in H^0\big(X_k,\oc_{X_k}(k')\otimes \pi_{0,k}^*L^{k+1}\big)
\end{align} to be the inverse image of the Wronskian $W(s_0,\ldots,s_k)$ under \eqref{local isomorphism}. 

Following \cite[\S 2.3]{Bro17},   define
\begin{align*}
\mathbb{W}(X_k,L):&={\rm Span}\{\omega(s_0,\ldots,s_n)\mid s_0,\ldots,s_n\in H^0(X,L) \}\\
&\subset H^0\big(X_k,\oc_{X_k}(k')\otimes \pi_{0,k}^* L^{k+1} \big)
\end{align*}
and  define the \emph{$k$-th Wronskian ideal sheaf} of $L$, denoted by $\mathfrak{w}(X_k,L)$, to be the base ideal  of  $\mathbb{W}(X_k,L)$. It was also shown in  \cite[\S 2.3]{Bro17} that
 if $L$ is very ample, one has  a chain of inclusions
$$
\mathfrak{w}(X_k,L)\subset \mathfrak{w}(X_k,L^2)\subset \cdots \subset \mathfrak{w}(X_k,L^m)\subset\cdots.
$$
By   noetherianity, this   increasing sequence stabilizes after some $m_{\infty}(X_k,L)\in \mathbb{N}$, and   the obtained asymptotic ideal sheaf is denoted by $
\mathfrak{w}_{\infty}(X_k,L).
$
Let us mention that $m_{\infty}(X_k,L)$ concerns   the first noetherianity argument in \cite{Bro17}, and in the rest of this subsection we will apply our new interpretation of Brotbek's Wronskians in \eqref{eq:factorization} to render  $m_{\infty}(X_k,L)$  effective.  The strategy is to compare   the globally defined  Wronskian ideal sheaves  $\{\mathfrak{w}(X_k,L^m)\}_{m\in \mathbb{N}}$ to the intrinsic ideal sheaf $\wk_{k,U}$.

One direction is easy to see from the very definition of $\mathfrak{w}(X_k,L)$.  By \eqref{eq:local wronskian}, for any $s_0,\ldots,s_k\in H^0(X,L)$, the Wronskian can be localized by
 \begin{align*} 
W(s_0,\ldots,s_k)_{\upharpoonright U}=W_{U}(s_{0,U},\ldots, s_{k,U})\cdot e_U^{k+1}\in   (E_{k,k'}T_X^*\otimes L^{k+1})(U).
\end{align*}
We denote by $\omega_U(s_{0,U},\ldots, s_{k,U})\in \oc_{X_k}(k')(U_k)$ the corresponding element of $W_{U}(s_{0,U},\ldots, s_{k,U})$ under the isomorphism \eqref{local isomorphism},  where $U_k:=\pi_{0,k}^{-1}(U)$.  In view of \eqref{local-pull-back}, one has $
\oc_{X_k}(k')(U_k)\simeq H^0(U, U\times F_{n,k} )
$, or more precisely,  
\begin{displaymath}
\xymatrix{  H^0(U, \Lambda^{k+1} J^k\oc_U)       \ar[d]^-{\Psi_U^{-1}}_{\rotatebox{90}{$\simeq$}} \ar[r]^-{W_{J^k{\oc_U}}} &  \ar[d]^-{\varphi_U^{-1}}_{\rotatebox{90}{$\simeq$}} H^0(U,E_{k,k'}T_U^*) \\
  H^0(U,U\times \Lambda^{k+1}\cb^{I_{n,k}}) \ar[r]^-{\mathds{1}_U\times \nu_{n,k} } & H^0(U, U\times F_{n,k} ). 
}
\end{displaymath}
By \eqref{eq:factorization},  $W_{U}(s_{0,U},\ldots, s_{k,U})=W_{J^k{\oc_U}}(j^ks_{0,U}\wedge \cdots \wedge j^ks_{k,U})$. Hence 
\begin{align}\label{eq:inverse iso}
\omega_U(s_{0,U},\ldots, s_{k,U})\simeq (\mathds{1}_U\times \nu_{n,k}) \circ \Psi_U^{-1}(j^ks_{0,U}\wedge \cdots \wedge j^ks_{k,U}).
\end{align}
Recall that  $\mathfrak{I}_{n,k} \subset \oc_{\rc_{n,k}}$ 
 is the base ideal of the linear system $|{\rm Im}(\nu_{n,k} )|$, and $\wk_{k,U}$ is defined to be the ideal sheaf ${\rm pr}_2^*(\mathfrak{I}_{n,k})$ on $U_k\simeq U\times \rc_{n,k}$. By \eqref{eq:inverse iso},  the base ideal of   $\omega_U(s_{0,U},\ldots, s_{k,U})$ is contained in $\wk_{k,U}$. As $s_0,\ldots,s_k\in H^0(X,L)$ are arbitrary, this leads to 
\begin{align}\label{inclusion}
 \wk(X_k,L)_{\upharpoonright U_k}\subset \wk_{k,U}.
\end{align}

Now we further assume that  the  line bundle $L$ separates $k$-jets everywhere, \emph{i.e.}  the $\cb$-linear map 
$$
H^0(X,L)\xrightarrow{j^k_L}  H^0(X,J^kL)\to J^kL_{x}
$$
is surjective for any $x\in X$. Then  
$$
\Lambda^{k+1}H^0(X,L)\to \Lambda^{k+1}\oc(U)\xrightarrow{j^k} \Lambda^{k+1}J^k{\oc_U}(U)\to \Lambda^{k+1}J^k\oc_{x}\simeq \Lambda^{k+1}\cb^{I_{n,k}}
$$
is also surjective for any $x\in U$.  By \eqref{eq:inverse iso} again,  $$\Im(\nu_{n,k})={\rm Span}\{\omega_U(s_{0,U},\ldots, s_{k,U})_{\upharpoonright \{x\}\times \rc_{n,k}} \mid s_0,\ldots,s_k\in H^0(X,L)\},$$
where we identify $\{x\}\times \rc_{n,k}$ with the fiber $\pi_{0,k}^{-1}(x)$. 
Write $\iota_x$ to be the composition $\rc_{n,k}\to \{x\}\times \rc_{n,k} \hookrightarrow  U\times \rc_{n,k}\xrightarrow{\simeq }U_k\hookrightarrow X_k$. This in turn implies that 
$$\iota_x^*\wk(X_k,L):=\iota_x^{-1}\wk(X_k,L)\otimes _{\iota_x^{-1}\oc_{X_k}}\oc_{\rc_{n,k}}=\mathfrak{I}_{n,k}$$
It follows from $\wk_{k,U}:={\rm pr}_2^*\mathfrak{I}_{n,k}$ that 
$ \wk(X_k,L)_{\upharpoonright U_k}= \wk_{k,U}$.  
By the inclusive relation \eqref{inclusion}, one has 
 \begin{align}\label{local}
 \wk_{k,U}=\wk(X_k,L)_{\upharpoonright U_k}=\wk(X_k,L^2)_{\upharpoonright U_k}=\cdots=\wk(X_k,L^k)_{\upharpoonright U_k}=\cdots.
 \end{align} 
As is well-known, $A^k$ separates $k$-jets everywhere once $A$ is very ample. By \eqref{local}, we conclude the following result.
\begin{thm}\label{wronskian result} 
	Let $X$ be a projective manifold, and let  $A$ be a very ample line bundle on $X$. Then $\wk(X_k,A^k)=\wk_{\infty}(X_k,A)$ and $m_{\infty}(X_k,A)=k$. 
\end{thm}
Moreover, it follows from the  relation \eqref{local}  that  the asymptotic Wronskian ideal sheaf is intrinsically defined, \emph{i.e.} $\wk_{\infty}(X_k,L)$ does not depend on  the very ample line bundle $L$. This reproves   \cite[Lemma 2.8]{Bro17}.  It also  allows us to denote by  $\wk_{\infty}(X_k)$ the asymptotic Wronskian ideal sheaf.

\begin{rem}
	In a joint work with Brotbek \cite{BD18}, we  generalize  the alternative interpretation of Wronskians by jets of sections of line bundles in this subsection to the logarithmic settings.
\end{rem}

\subsection{Blow-up  of the Wronskian ideal sheaf}\label{blow-up}
This subsection is mainly borrowed from \cite{Bro17}. We will state some important results without proof, and the readers who are interested in further details are encouraged to refer to \cite[\S 2.4]{Bro17}. Let us first recall the following crucial property of the Wronskian ideal sheaf in \cite{Bro17}.
\begin{lem}[\!\!\protect{\cite[Lemma 2.4]{Bro17}}]
Let $X$ be a  projective manifold equipped with a very ample line bundle $L$.  Then 
	$$
	{\rm Supp}\big(\oc_{X_k}/\mathfrak{w}(X_k,L^k)\big)\subset X_k^{\rm sing},
	$$
	where $ X_k^{\rm sing}$ is the set of singular $k$-jets of $X_k$.
\end{lem}  
Based on the above lemma, as was shown in  \cite{Bro17}, Brotbek introduced a \emph{fonctorial birational morphism} of the Demailly-Semple $k$-jet  tower $\nu_k:\hat{X}_k\to X_k$ by blowing-up the asymptotic Wronskian ideal sheaf $\wk_{\infty}(X_k)$, so that he was able to establish a \emph{strong Zariski open property}  for hyperbolicity. Indeed,  Brotbek even built the  {strong Zariski open property}  for  almost  $k$-jet ampleness. We require the following results in \cite{Bro17} to proceed further.

\begin{thm}[\!\!\protect{\cite[Propositions 2.10, 2.11 and 2.13]{Bro17}}]\label{fonctorial}
	Let $X$ be a  projective manifold. 
	\begin{thmlist} 
	\item For any smooth closed submanifold $Y\subset X$, the inclusion $Y_k\subset X_k$ induces an inclusion
 $
\hat{Y}_k\subset \hat{X}_k
$. Moreover, $\hat{Y}_k$ is the strict transform of $Y_k$ in $\hat{X}_k$. 
	\item\label{Zariski}  If 
	\begin{equation}\label{eq:star}\tag{$\ast$}
	\exists a_0,\ldots,a_k\in \mathbb{N} \quad {\rm s.t.} \quad 	\nu_k^* \oc_{X_k}(a_k,\ldots,a_1) \otimes \oc_{\hat{X}_k}(-a_0F)\ \ {\rm is\ ample},
	\end{equation}
	then $X$ is \emph{almost $k$-jet ample}. Here $F$ is an effective divisor on $\hat{X}_k$ defined by $\oc_{\hat{X}_k}(-F)=\nu_k^*\wk_{\infty}(X_k)$.
	\item Let $\xs\xrightarrow{\rho} T$ be a smooth and projective morphism between non-singular varieties. We denote by $\xs^{\rm rel}_{k}$ the  Demailly-Semple $k$-jet  tower of the relative directed variety $(\xs,T_{\xs/T})$. Take $\nu_k:\hat{\xs}_k^{\rm rel}\rightarrow \xs_k^{\rm rel}$ to be the blow-up of the asymptotic Wronskian ideal sheaf $\wk_{\infty}(\xs_k^{\rm rel})$. Then for any $t_0\in T$ writing $X_{t_0}:=\rho^{-1}(t_0)$, we have
	$
	\nu_k^{-1}(X_{t_0,k})=\hat{X}_{t_0,k}
	$. 
	\item 	Property \eqref{eq:star} is a Zariski open property. Precisely speaking, in the same setting as above, if there exists \(t\in T\) such that   \(X_{t}\) satisfies  \eqref{eq:star},  then there exists a non-empty Zariski open subset \(T_0\subset T\) such that for any \(s\in T_0\),  \(X_s\) satisfies \eqref{eq:star} as well. In particular, $X_s$ is almost $k$-jet ample for all $s\in T_0$.
\end{thmlist}
\end{thm}

\section{An effective  Nakamaye's theorem}\label{Nakamaye}
As mentioned in \cref{intro}, both \cite{BD15} and \cite{Bro17} applied  Nakamaye's Theorem on the augmented base locus \cite{Nak00} for  families of  zero-dimensional subschemes  to  provide  a
geometric control on base locus.  In this section  we   render their noetherianity arguments effective. 
\medskip

We start by setting notations as in \cite[\S 3]{Bro17}.  Consider $V:=H^0\big(\pb^N,\oc_{\pb^N}(\delta)\big)$,  which can be identified with  the space of homogeneous polynomials of degree $\delta$ in $\cb[z_0,\ldots,z_N]$. For any $J\subset \{0,\ldots,N\}$  we set
$$
\pb_J:=\{[z_0,\ldots,z_N]\in \pb^N \mid  z_j=0\ {{\rm if}}\ j\in J \}.
$$
Given any $\Delta\in {\rm Gr}_{k+1}(V)$ and $[z]\in \pb^N$, we denote by $\Delta([z])=0$ once   $P([z])=0$ for any  $P\in \Delta\subset V$.  
Define the \emph{universal family of complete intersections} to be
\begin{eqnarray}\label{universal}
\mathscr{Y}:=\{(\Delta,[z])\in {{\rm Gr}}_{k+1}(V)\times \pb^N \mid \Delta([z])=0 \}.
\end{eqnarray}
For any $J\subset\{0,\ldots,N\}$, set 
\begin{align}\label{univj}
\mathscr{Y}_J:=\mathscr{Y}\cap ({{\rm Gr}}_{k+1}(V)\times \pb_J).
\end{align} 
Let us denote by   $p:\ys\to {{\rm Gr}}_{k+1}(V)$ and $q:\ys\to \pb^N$ the   projection maps.   The next lemma is our starting point.
\begin{lem}\label{smooth grassmannian}
	For any $J\subset\{0,\ldots,N\}$, $\ys_J\rightarrow \pb_J$ is a locally  trivial holomorphic  fibration with fibers isomorphic to the Grassmannian ${\rm Gr}\big(k+1,{\rm dim}(V)-1\big)$. In particular, $\ys_J$ is a smooth projective manifold.
\end{lem}
\begin{proof}
	Any linear transformation $g\in {\rm GL}_{N+1}(\cb)$ induces a natural action $\tilde{g}\in {\rm GL}(V)$, hence also induces a biholomorphism $\hat{g}$ of ${\rm Gr}_{k+1}(V)$. Observe that for any $[z]\in \pb^N$, $\hat{g}$ maps the fiber $q^{-1}([z])$ to $q^{-1} ([g\cdot z]) $ bijectively. Since ${\rm GL}_{N+1}(\cb)$ acts transitively on $\pb^N$, the fibration $q:\ys\rightarrow\pb^N$ can thus  be  trivialized locally. 
	
	Take a special point $[\mathbf{e}_0]:=[1,0,\ldots,0]\in \pb^N$. For any $P=\sum_{|I|=\delta}a_Iz^{I}\in V$,  $P([\mathbf{e}_0])=0$ if and only if the coefficient of $z_0^{\delta}$ in  $P$ is zero. If we denote by $V_0$ the subspace of $V$ spanned by $\{z^I\mid |I|=\delta, z^I\neq z_0^\delta \}$, then $q^{-1}([\mathbf{e}_0])={\rm Gr}_{k+1}(V_0)\simeq {\rm Gr}\big(k+1,{\rm dim}(V)-1\big)$. The lemma is thus proved.
\end{proof}
Observe that when $k+1\geqslant N$,  $p:\ys\rightarrow {\rm Gr}_{k+1}(V)$ is a \emph{generically finite to one} morphism. 
Let us denote by $\ls$ be the very ample line bundle on $\grs$ which is the pull back of $\oc(1)$ on ${\rm P}(\Lambda^{k+1}V)$ under the Pl\"ucker embedding $\grs\hookrightarrow  {\rm P}(\Lambda^{k+1}V)$. Then $p^*\ls_{\upharpoonright \ys_J}$ is a big and nef line bundle on $\ys_J$ for any $J\subset \{0,\ldots,N\}$. Write $p_J:\ys_J\rightarrow \grs$ and $q_J:\ys_J\rightarrow \pb_J$ for the natural projections, and define
\begin{align*}
E_J&:=\{y\in\ys \mid {{\rm dim}}_y\big(p_J^{-1}(p_J(y))\big)>0 \}\\
G^\infty_J&:=p_J(E_J)\subset \grs. 
\end{align*}
When $J=\varnothing $ we simply write $E:=E_{\varnothing }$ and $G^\infty:=G^\infty_{\varnothing }$.
 By the definition of \emph{null locus} \cite[Definition 10.3.4]{Laz04II},   $E_J={\rm Null}(p_J^*\ls)$.   It then follows from Nakamaye's theorem \cite{Nak00} that 
 $$
\mathbf{B}_+(p_J^*\ls)={\rm Null}(p_J^*\ls)=E_J. 
 $$
 Observe that the line bundle $\ls\boxtimes \oc_{\pb^N}(1)$ on $\grs\times \pb^N$ is ample, and so is its restriction to $\ys_J$.  Hence by the definition of augmented base locus and noetherianity, there exists $m_J\in \mathbb{N}$ such that
 \begin{align}\label{eq:nakamaye}
 {\rm Bs}\big( \ls^m\boxtimes \oc_{\pb_J}(-1)_{\upharpoonright \ys_J}\big)=\mathbf{B}_+(p_J^*\ls)=E_J\subset p_J^{-1}(G^\infty),  \quad \forall m\geqslant m_J.
 \end{align}
 We emphasize that  the value $M:={\rm max}\{m_J\mid J\subset \{0,\ldots,N\} \}$ concerns the second noetherianity argument in \cite{Bro17} resulting in the loss of effective lower degree bounds  $d_{{\rm Kob},n}$ in \cref{Brotbek}.  
 
 Instead of requiring \eqref{eq:nakamaye}, we will provide a slightly weaker base control but with an effective estimate on $M$, which still remains valid in  Brotbek's proof (see \cite[Remark 3.13]{Bro17}).
 \begin{thm}\label{weak nakamaye}
When $m\geqslant \delta^k$, for any  $J\subset \{0,\ldots,N\}$, one has
\begin{align}\label{eq:nakamaye valid}
{\rm Bs}\big( \ls^m\boxtimes \oc_{\pb_J}(-1)_{\upharpoonright \ys_J}\big) \subset p_J^{-1}(G_J^\infty).
\end{align}
 \end{thm}
 

To prove \cref{weak nakamaye}, we construct  \emph{sufficiently many} global sections of $\ls^m\boxtimes \oc_{\pb_J}(-1)_{\upharpoonright \ys_J}$ in an explicit manner to control their base locus. Precisely speaking,  for any $\Delta\notin G^{\infty}_{J}$, by definition $p_J^{-1}(\Delta)$ is a finite set. We will show that  for each $m\geqslant \delta^{k}$  there exists  an effective divisor $D_\Delta\in |\ls^m\boxtimes \oc_{\pb_J}(-1)_{\upharpoonright \ys_J}|$ so that $D_\Delta\cap p_J^{-1}(\Delta)=\varnothing$.


Let us first recall a version of  \emph{projection formula} in \emph{intersection theory}, which is  indeed a direct consequence of  \cite[Example 8.1.7]{Ful13}.\noindent
\begin{thm}[Projection formula]\label{projection formula}
	Let $f:X\rightarrow Y$ be a generically finite to one and surjective morphism between non-singular irreducible varieties, and $x$ (resp. $y$) be cycle on $X$ (resp. $Y$) of dimension $k$ (resp. ${\rm dim}(X)-k$). Then 
	$$
	{\rm deg}\Big(f_*\big(f^*(y)\cdot x\big)\Big)={\rm deg}\big(y\cdot f_*(x)\big),
	$$
	where $f^*$ and $f_*$ are defined in the Chow group. When the scheme-theoretic inverse image $f^{-1}(y)$ is of pure dimension ${\rm dim}(X)-k$, one has $f^*(y)=[f^{-1}(y)]$.
\end{thm}	

\begin{proof}[Proof of \cref{weak nakamaye}]
	We first deal with the case $k+1=N$, and  then reduce the general setting $k+1\geqslant  N$ to this case.
 \begin{claim}\label{claim:non-vanishing}
 When $k+1=N$, 	$H^0\big(\ys,   \ls^m\boxtimes \oc_{\pb^N}(-1)_{\upharpoonright \ys}\big) \neq \varnothing$   for all  $m\geqslant  \delta^{N-1}$.
 	\end{claim}
 \begin{proof}   Let us pick a smooth curve $C$ in ${\rm Gr}_{N}(V)$ of degree $1$ with respect to $\ls$,  given by 
$$
\Delta([t_0,t_1]):={{\rm Span}}(z_1^\delta, z_2^\delta,\ldots z_{N-1}^\delta, t_0z_N^\delta+t_1z_0^\delta),
$$
where $[t_0,t_1]\in \pb^1$. Indeed, the curve $C$ is the line in the Pl\"ucker embedding ${\rm P}(\Lambda^{N}V)$ defined by two vectors $z_{1}^{\delta}\wedge \cdots \wedge z_{N-1}^{\delta}\wedge z_0^{\delta}$ and $z_1^{\delta}\wedge   \cdots \wedge z_{N}^{\delta}$ in $\Lambda^{N}V$. Hence   $\ls\cdot C=1$.

Consider a  hyperplane $D$ in $\pb^N$ given by  $\{[z_0,\ldots,z_N]\mid z_0+z_N=0\}$. Since $p:\ys\to {\rm Gr}_N(V)$ is a generically finite to one and surjective morphism, $p_*q^*D$ is an effective  divisor in ${\rm Gr}_N(V)$.  
 	
	Since $p^{-1}(C)$ has pure dimension $1$, then $p^*C$ is a 1-cycle in $\ys$.
	An easy computation shows that $p^*C$ and $q^*D$ intersect only at the point 
	$${{\rm Span}}(z_1^\delta, z_2^\delta,\ldots z_{N-1}^\delta,\\ z_N^\delta+(-1)^{\delta+1}z_0^\delta)\times [1,0,\ldots,0,-1]\in \ys$$
	with multiplicity $\delta^{N-1}$. Hence $p^*C\cdot q^*D=\delta^{N-1}$. By \cref{projection formula}, one has
	\begin{align*}
		C\cdot p_*(q^*D)=p_* (p^*C \cdot q^*D )=\delta^{N-1}.
	\end{align*}
Note that the Picard group ${\rm Pic}\big({\rm Gr}_{N}(V)\big)\simeq \zbb$ is generated by  $\ls$, which in turn implies 
\begin{align}\label{intersection}
p_*q^*D\in  |\ls^{\delta^{N-1}}|
\end{align} by the fact that $\ls\cdot C=1$.
It follows from \cref{smooth grassmannian} that $q^*D$ is a  smooth hypersurface in $\ys$. Since ${\rm Supp}(q^*D)\subset {\rm Supp}(p^*p_*q^*D)$, $p^*p_*q^*D-q^*D$ is thus an effective divisor  of $\ys$, and by \eqref{intersection}
\begin{align}\label{intersection2}
p^*p_*q^*D-q^*D\in |\ls^{\delta^{N-1}}\boxtimes \oc_{\pb^N}(-1)_{\upharpoonright \ys}|.
\end{align}
The claim follows from the fact that $\ls$ is very ample. 
\end{proof}
The base locus of $|\ls^{\delta^{N-1}}\boxtimes \oc_{\pb^N}(-1)_{\upharpoonright \ys}|$ can be well understood.
\begin{claim}\label{claim:base control}
	For any $m\geqslant \delta^{N-1}$, the base locus
	\begin{eqnarray}\label{claim}
	{{\rm Bs}}\big(\ls^{m}\boxtimes \oc_{\pb^N}(-1)_{\upharpoonright \ys}\big)\subset p^{-1}(G^{\infty}).
	\end{eqnarray}
\end{claim}
\begin{proof}
	For given any $\Delta_0\notin G^{\infty}$,    $ p^{-1}(\Delta_0)$ is a finite set by the definition of $G^{\infty}$.  Then one can take  a general hyperplane $D\in |\oc_{\pb^N}(1)|$ such that $D\cap q\big(p^{-1}(\Delta_0)\big)=\varnothing $.  By \eqref{intersection2}, $D$ gives rise to an effective divisor $$p^*p_*q^*D-q^*D\in |\ls^{\delta^{N-1}}\boxtimes \oc_{\pb^N}(-1)_{\upharpoonright \ys}|.$$
	For any $\Delta\in {\rm Gr}_{N}(V)$, if we denote by $${\rm Int}(\Delta):=\{[z]\in\pb^N \mid  \Delta([z])=0\},$$
	then $q\big(p^{-1}(\Delta)\big)={{\rm Int}}(\Delta)$. Hence the condition   $D\cap q\big(p^{-1}(\Delta_0)\big)=\varnothing $ is equivalent to that ${{\rm Int}}(\Delta_0)\cap D=\varnothing $. On the other hand, for any $\Delta\in {\rm Supp}(p_*q^*D)$,  one has  ${{\rm Int}}(\Delta)\cap D\neq \varnothing $, and thus we conclude that  $\Delta_0\notin {\rm Supp}(p_*q^*D)$.  In particular, 
	$$
	p^{-1}(\Delta_0) \cap  {\rm Supp} (p^*p_*q^*D-q^*D)=\varnothing.
	$$
	As $\Delta_0$ is an arbitrary point outside $G^{\infty}$, we conclude that
	$$
	{{\rm Bs}}\big(\ls^{\delta^{N-1}}\boxtimes \oc_{\pb^N}(-1)_{\upharpoonright \ys}\big)\subset p^{-1}(G^{\infty}).
	$$
	Since $\ls$ is very ample,   we have
	$$
	{{\rm Bs}}\big(\ls^{m}\boxtimes \oc_{\pb^N}(-1)_{\upharpoonright \ys} \big)\subset 	{{\rm Bs}}\big(\ls^{\delta^{N-1}}\boxtimes \oc_{\pb^N}(-1)_{\upharpoonright \ys} \big)\subset p^{-1}(G^{\infty})
	$$
	for any $m\geqslant \delta^{N-1}$. 
	The claim is thus proved.
\end{proof}

Let us  deal with the general case $J\supsetneq \varnothing$. Without loss of generality we can assume that $J=\{n+1,\ldots,N\}$.  
For any $\Delta_0\in p_J(\ys_J)\setminus G_J^{\infty}$,   the set $ p_J^{-1}(\Delta_0)={\rm Int}(\Delta_0)\cap \pb_J$ is finite.  We can also take a general hyperplane $D\in  |\oc_{\pb^N}(1)|$ such that ${{\rm Int}}(\Delta_0)\cap D\cap \pb_J=\varnothing $. One can further choose a proper coordinate for $\pb^N$ such that $D=(z_n=0)$.

By \cref{smooth grassmannian}, $q_J^{*}(D\cap \pb_J)$ is a  smooth hypersurface  in $\ys_J$. Set $F:=p_J\big(q_J^{-1}(D\cap \pb_J)\big)$ set-theoretically.  Then for any effective  divisor $\tilde{H}\in |\ls^m|$ on ${\rm Gr}_N(V)$ such that $F\subset {\rm Supp}(\tilde{H})$ and $p_J(\ys_J)\not\subset {\rm Supp}(\tilde{H})$,  
\begin{align}\label{eq:effective}
p_J^*(\tilde{H})-q_J^{*}(D\cap \pb_J)\in | \ls^m\boxtimes  \oc_{\pb^N}(-1)_{\upharpoonright \ys_J}|
\end{align}
is an effective divisor of $\ys_J$.  However,  it may happen that for any hyperplane $\tilde{D}\in  |\oc_{\pb^N}(1)|$, all  constructed divisors of the form $p_*q^*(\tilde{D})$ will always contain $\Delta_0$.

\medskip

Choose a decomposition of $V=V_1\oplus V_2$ such that $V_1$ is spanned by the vectors $\{z^\alpha\in V \mid  \alpha_n=\cdots=\alpha_N=0\}$ and $V_2$ is spanned by other $z^\alpha$'s. Let us denote $G$ to be the subgroup of the general linear group $GL(V)$ which is the lower triangle matrix with respect to the decomposition of $V=V_1\oplus V_2$ as follows:
\begin{align}\label{matrix}
G:=\Big\{g\in GL(V)\mid g= 
\begin{bmatrix}
I & 0\\
A & B
\end{bmatrix}, B\in GL(V_2), A\in {\rm Hom}(V_1,V_2)
\Big\}.
\end{align}
The subgroup $G$ also induces  a natural group action on the Grassmannian ${{\rm Gr}}_{N}(V)$, and  we have the following
\begin{claim}\label{action}
	Set $H:=p_*(q^*D)$. Then for any $g\in G$, $F\subset g(H)$ and there exists a $g_0\in G$ such that $\Delta_0\notin g_0(H)$.
\end{claim}
\begin{proof}
	For any $\Delta\in {{\rm Gr}}_{N}(V)$, choose $\{s_1,\ldots,s_N\}\subset V$ which spans $\Delta$. Let $s_i=u_i+v_i$ be the unique decomposition of $s_i$ under $V=V_1\oplus V_2$.  Recall that $F:=p_J\big(q_J^{-1}(D\cap \pb_J)\big)$.    Then   
	\begin{align}\label{eq:equivalent}
	\Delta\in F  \iff \cap_{i=1}^N(u_i=0)\cap \pb^{n-1}\neq \varnothing,
	\end{align} where $\pb^{n-1}:=\{[z_0,\ldots,z_N]\in \pb^N\mid  z_j=0 \mbox{ for } j\geqslant n \}=D\cap \pb_J$, and we can identify $V_1$ with $H^0\big(\pb^{n-1},\oc_{\pb^{n-1}}(\delta)\big)$.
	
	For any $g\in GL(V)$, $g(\Delta)$ is spanned by $\{g(s_1),\ldots,g(s_N)\}$. By the definition of $G$, for any $g\in G$, we have the decomposition $g(s_i)=u_i+v_i'$ with respect to $V=V_1\oplus V_2$ which keeps the $V_1$ factors invariant.  Then $g(F)= F$ for any $g\in G$ by \eqref{eq:equivalent}.  The first statement   follows from the fact  $F\subset H$.

Now we take $\{t_1,\ldots,t_N\}\subset V$ which spans $\Delta_0$. Denote  $t_i=u_i+v_i$ to be the   decomposition of $t_i$ under $V=V_1\oplus V_2$. By our choice of $D$, ${\rm Int}(\Delta_0)\cap \pb^{n-1}=\varnothing $,  which is equivalent to $\cap_{i=1}^N(u_i=0)\cap \pb^{n-1}= \varnothing $ by \eqref{eq:equivalent}.   We can then choose the proper basis $\{t_1,\ldots,t_N\}$ spanning $\Delta_0$, so that
	\begin{thmlist} 
		\item $\cap_{i=1}^{n}(u_i=0)\cap \pb^{n-1}=\varnothing $;
		\item for some $m\geqslant n$, $\{u_1,\ldots,u_m\}$ is a set of vectors in $V_1$ which is linearly independent;
		\item $u_{m+1}=\cdots=u_N=0$.
	\end{thmlist} 
	Then $\cap_{i=1}^{n}(u_i=0)\cap \{z_n=0\}=\pb^{N-n-1}:=\{[z_0,\ldots,z_N]\in \pb^N\mid z_j=0 \mbox{ for } j\leqslant n \}$, and $\{v_{m+1},\ldots, v_N\}$ is a set of linearly independent vectors in $V_2$.
	
	Take a point  $\Delta'\in {{\rm Gr}}_{N}(V)$ spanned by
	\begin{align*}
		\begin{cases}
			\tilde{t}_{1}:=u_{1}\\
			\vdots\\
			\tilde{t}_{n}:=u_{n}\\
			\tilde{t}_{n+1}:=u_{n+1}+z_{n+1}^{\delta}\\
			\vdots\\
			\tilde{t}_{m}:=u_{m}+z_{m}^{\delta}\\
			\tilde{t}_{m+1}:=u_{m+1}+z_{m+1}^{\delta}=z_{m+1}^{\delta}\\
			\vdots\\	
			\tilde{t}_N:=u_N+z_N^{\delta}=z_N^{\delta}\\
		\end{cases}.
	\end{align*}
	Then one can easily observe that ${{\rm Int}}(\Delta')\cap  (z_n=0)=\varnothing $, which is equivalent to that $\Delta'\notin H=p_*q^*(D)$.  We will find a $g_0\in G$ such that $g_0(\Delta')=\Delta_0$. 
	
	Indeed, since $\{v_{m+1},\ldots, v_N\}\subset V_2$  and $\{u_1,\ldots,u_m\}\subset V_1$ are both linearly independent, we can find a $B\in GL(V_2)$ such that $B(z^\delta_i)=v_i$ for all $i\geqslant m+1$,  and $A\in \mbox{Hom}(V_1,V_2)$ satisfying that 
	\begin{align*}
	\begin{cases}
A(u_i)=v_i\quad  &{\rm for}\quad  1\leqslant i\leqslant n, \\
A(u_j)=v_j-B(z_j^\delta)\quad   &{\rm for} \quad  n+1\leqslant j\leqslant m.
	\end{cases}
	\end{align*}
	Set  $g_0:= 
	\begin{bmatrix}
	I & 0\\
	A & B
	\end{bmatrix}$ which is of the type \eqref{matrix}. We have
	$$g_0(\Delta')={\rm Span}\{g_0(\tilde{t}_1),\ldots,g_0(\tilde{t}_N)\}={\rm Span}\{t_1,\ldots,t_N\}= \Delta_0.$$ 
	Recall that  $\Delta'\notin H$.  Then  $\Delta_0\notin g_0(H)$ and we finish the proof of the claim.
\end{proof}

Since $H\in |\ls^{\delta^{N-1}}|$ by \eqref{intersection}, we claim that $g_0(H)\in |\ls^{\delta^{N-1}}|$. Indeed, since the complex general linear group $GL(V)$ is connected, the biholomorphism  of ${\rm Gr}_{N}(V)$ induced by $g_0\in GL(V)$ is homotopic to the identity map, and thus $H$ and $g_0(H)$ lie on the same linear system. By \cref{action}, $F\subset g_0(H)$ and $\Delta_0\notin g_0(H)$.  By \eqref{eq:effective}, the divisor 
$$p_J^*\big(g_0(H)\big)-q_J^{*}(D\cap \pb_J)\in | \ls^{\delta^{N-1}}\boxtimes   \oc_{\pb^N}(-1)_{\upharpoonright \ys_J}|$$ 
is effective and avoids the finite set $p_J^{-1}(\Delta_0)$.

Note that $\Delta_0\in {\rm Gr}_{N}(V)$ is an arbitrary point in $p_J(\ys_J)\setminus G_J^{\infty}$. This in turn  proves \cref{weak nakamaye} for the case  $k+1=N$.  

\medskip

Let us show how to deal with the general cases $k+1>N$.  

For any $J\subset \{0,\ldots,N\}$, one can see $\pb_J\subset \pb^N$ as subspaces of $\pb^{k+1}$ defined by 
\begin{align*}
{\pb^N}&:=\{ [z_0,\ldots,z_{k+1}] \in \pb^{k+1}  \mid z_{N+1}=\cdots=z_{k+1}=0\},\\
\pb_{{J}}&:=\big\{[z_0,\ldots,z_{k+1}]\in \pb^{k+1} \mid z_j=0\ {{\rm if}}\ j\in J\cup \{N+1,\ldots,k+1\} \big\}.
\end{align*}
Set $V_k:=H^0\big(\pb^{k+1},\oc_{\pb^{k+1}}(\delta)\big)$,  
and 
$$
\tilde{\ys}_J:=\{(\Delta,[z])\in {{\rm Gr}}_{k+1}(V_k)\times \pb_{{J}}\mid  \Delta([z])=0 \}.
$$
 There is a natural inclusion ${\rm Gr}_{k+1}(V)\subset {\rm Gr}_{k+1}(V_k)$.  
Define $\tilde{p}_J:\tilde{\ys}_J\rightarrow {\rm Gr}_{k+1}(V_k)$ and $\tilde{q}_J:\tilde{\ys}_J\rightarrow \pb_{{J}}$ to be the natural   projections. Set 
$$\tilde{G}_J^{\infty}:=\{\Delta\in {{\rm Gr}}_{k+1}(V_k)\mid \tilde{p}_J^{-1}(\Delta) \mbox{ is not finite set}\}.$$
Hence by the above arguments, for $m\geqslant \delta^{k}$,  we  have
\begin{align}\label{eq:inclusion}
 {{\rm Bs}}(\ls_k^m\boxtimes \oc_{\pb^{k+1}}(-1)_{\upharpoonright \tilde{\ys}_{{J}}}) 
 \subset  \tilde{p}_{J}^{-1}(\tilde{G}_{J}^{\infty}), \end{align}
 where $\ls_k$ is the tautological line bundle  on ${\rm Gr}_{k+1}(V_k)$.  
  
Recall that  $\ys\subset {\rm Gr}_{k+1}(V)\times \pb^N$ and $\ys_{{J}}\subset {\rm Gr}_{k+1}(V)\times \pb_J$ are  the universal families of complete intersections defined in \eqref{universal} and \eqref{univj}. The inclusion  $\iota_k:{\rm Gr}_{k+1}(V)\hookrightarrow{\rm Gr}_{k+1}(V_k)$  induces the following inclusions 
$$ \begin{tikzcd}
{\ys}_J  \ar[d,hook]  \ar[r,hook] & {\rm Gr}_{k+1}(V)\times \pb_{{J}} \ar[d, hook, "\iota_k\times \mathds{1}"]      \\
	\tilde{\ys}_J  \ar[r,hook] &{\rm Gr}_{k+1}(V_k)\times\pb_{{J}} 
\end{tikzcd}
$$  
Observe that  $
{G}_J^{\infty}=\tilde{G}_{{J}}^{\infty}\cap {\rm Gr}_{k+1}(V)
$.   Note that   
$\iota_k^*\ls_k:=\ls$, which  is still the tautological line bundle on ${\rm Gr}_{k+1}(V)$.  
Hence by the above arguments, for $m\geqslant \delta^{k}$,  we  have
\begin{align*} 
{{\rm Bs}}( \ls^m\boxtimes \oc_{\pb^{N}}(-1)_{\upharpoonright {\ys}_J})&= {{\rm Bs}}( \ls_k^m\boxtimes \oc_{\pb^{k+1}}(-1)_{\upharpoonright {\ys}_J})\\ 
&\subset  {{\rm Bs}}(\ls_k^m\boxtimes \oc_{\pb^{k+1}}(-1)_{\upharpoonright \tilde{\ys}_{{J}}})\cap  {\ys}_J\\ 
&\subset  \tilde{p}_{J}^{-1}(\tilde{G}_{J}^{\infty})\cap  {\ys}_J \quad\quad\quad\quad \big({\rm by} \  \eqref{eq:inclusion}\big) \\  
&=   {p}_J^{-1}\big ( {G}_J^{\infty} \big),
\end{align*}
where $p_{J}:\ys_{J}\rightarrow {\rm Gr}_{k+1}(V)$ and $q_{J}:\ys_{J}\rightarrow \pb_{{J}}$ are the projection maps. 
This in turn proves  \cref{weak nakamaye} for  the general cases $k+1\geqslant N$. 
\end{proof}

\begin{rem}
	Let us mention that 
	the  proof of \cref{weak nakamaye} is indeed constructive, and we do not rely on the general results by Nakamaye. 
\end{rem}

Now we are able to apply \cref{wronskian result,weak nakamaye} to prove \cref{main:Kobayashi}  using the explicit  formula of $d_{{\rm Kob},n}$ in \cite{Bro17}. 
\begin{proof}[Proof of \cref{main:Kobayashi}]
In \cite[p. 18]{Bro17}, Brotbek obtained the following formula
$$
d_{{\rm Kob},n}=  m_{\infty}(X_k,\as)+\delta+(R+k)\delta,
$$
where  $R:=M(k+1)\big(m_{\infty}(X_k,\as)+\delta-1+k\delta\big)+1$ with $M\in \mathbb{N}$  the lower bound of $m$ so that \eqref{eq:nakamaye valid} remains valid, and one can take $k=n-1$,  $\delta=n^2$ by \cite{Bro17}.    By \cref{wronskian result,weak nakamaye}, we can take  $m_{\infty}(X_k,\as)=k=n-1$,   and $M=\delta^k=\delta^{n-1}$. Hence
\begin{align*}
d_{{\rm Kob},n}&\leqslant  m_{\infty}(X_k,\as)+\delta+(R+k)\delta\\
&=k+\delta+\delta\Big( \delta^k(k+1)\big(k+\delta-1+k\delta\big)+1+k\Big)\\
&= n^{2n+1}(n^3+n-2)+n^3+n^2+ n-1\\
&\leqslant n^{2n+3}(n+1),
\end{align*} 
and the theorem follows.
\end{proof}
 \begin{rem} Along Siu's line of slanted vector fields on higher jet spaces outlined in   \cite{Siu04},      Diverio-Merker-Rousseau \cite{DMR10} first  proved the  \emph{weak hyperbolicity} (say that a projective  variety $X$ is weakly hyperbolic if all  entire curves lie in a proper subvariety $Y\subsetneq X$)   of general hypersurfaces in $\pb^n$ of degree $d\geqslant 2^{(n-1)^5}$. This lower bound was improved by  Demailly \cite{Dem11} to     $d\geqslant\floor[\Big]{\frac{n^4}{3}\Big(n\log\big(n\log(24n)\big)\Big)^n}$, and the latest best known bound $d\geqslant (5n)^2n^n$ 
 	was obtained by  Darondeau   \cite{Dar15}.  Very recently, Demailly \cite{Dem18} gave  a simple proof of the Kobayashi conjecture as well as an effective lower bound $d_{{\rm Kob},n}= \frac{1}{5}\big(e(n-1)\big)^{2n}$ for the degrees.\footnote{After the submission of the final version of the present paper, there are some new progress on the Kobayashi conjecture by Riedl-Yang \cite{RY18}. Based on the result in \cite{RY18} and the previous work by Darondeau \cite{Dar15}, Merker \cite{Mer18} was able to slightly improve  the effective bounds in \cref{main:Kobayashi}.}   
\end{rem}

\medskip

Now we will generalize \cref{weak nakamaye}  to the cases of products of Grassmannians.  Let us  fix  $c,k,n\in  \mathbb{N}$ with $c(k+1) \geqslant n$.    Write  $V_{\delta_i}:=H^0\big(\pb^n,\oc_{\pb^n}(\delta_i)\big)$ and $\gf :=\prod_{i=1}^{c}{\rm Gr}_{k+1}(V_{\delta_i})$ for any $(\delta_1,\ldots,\delta_c)\in \mathbb{N}^c$.  Set $\ys$   to be the \emph{universal family of complete intersections} defined by 
\begin{align}\label{eq:generalized universal}
\ys:=\{(\Delta_1,\ldots,\Delta_c,[z])\in \gf\times \pb^n \mid  \Delta_i([z])=0, \forall\ i=1,\ldots,c\}.
\end{align}
Denote by $p:\ys \rightarrow \gf$ and $q:\ys \rightarrow \pb^n$     the   projection maps.  Then $p$ is a generically finite to one   morphism. 	Define a group homeomorphism
\begin{align}\label{eq:group}
	\ls:\zbb^c&\rightarrow {{\rm Pic(\gf)}}\\\nonumber
	\mathbf{a}=(a_1,\ldots,a_c)&\mapsto \oc_{{\rm Gr}_{k+1}(V_{\delta_1})}(a_1)\boxtimes \cdots \boxtimes \oc_{{\rm Gr}_{k+1}(V_{\delta_c})}(a_c)
\end{align}
which is indeed an isomorphism.

Let us  introduce  $c$-smooth curves  $C_1,\ldots,C_c$ on $\gf$,  defined by 
\begin{eqnarray*}
	\Delta_i([t_0,t_1]):={{\rm Span}}(z_1^{\delta_1},z_{c+1}^{\delta_1},\ldots, z_{kc+1}^{\delta_1})\times {{\rm Span}}(z_2^{\delta_2},z_{c+2}^{\delta_2},\ldots, z_{kc+2}^{\delta_2})\times \cdots\\
	\times {{\rm Span}}(t_0z_{i}^{\delta_i}+t_1z_{0}^{\delta_i},z_{c+i}^{\delta_i},\ldots,z_{kc+i}^{\delta_i})\times \cdots \times{{\rm Span}}(z_{c}^{\delta_{c}},z_{2c}^{\delta_{c}},\ldots,z_{(k+1)c}^{\delta_{c}})
\end{eqnarray*}
for $[t_0,t_1]\in \pb^1$. It is easy to verify that  $\ls(\mathbf{a})\cdot C_i=a_i$ for each $i$. 
Consider the hyperplane $D_i\in |\oc_{\pb^n}(1)|$ given by  $\{[z_0,\ldots,z_n]\mid   z_{i}+z_0=0\}$. Then we have the similar result   as \cref{claim:non-vanishing}.
\begin{lem}
Suppose that $n=k(c+1)$.  	For any hyperplane $D\in |\oc_{\pb^n}(1) |$,   $p_{*}q^*D\in | \ls(\mathbf{b}) |$, where $\mathbf{b}:=(b_1,\ldots,b_c)\in \mathbb{N}^c$ with $b_i:=\frac{\prod_{j=1}^{c}\delta_j^{k+1}}{\delta_i}$.
\end{lem}
\begin{proof}
	It is easy to show that $p^*C_i$ and $q^*D_i$ intersect only at one point with multiplicity $b_i$ for each $i=1,\ldots,c$. By the projection formula in  \cref{projection formula},  one has
	\begin{align}\label{eq:multiplicity}
		(p_{*}q^*D_i)\cdot C_i=p_*(q^*D_i\cdot p^*C_i)=b_i.
	\end{align}
Recall that  $
	\ls(\af)\cdot C_i=a_i
	$   	for any $\af\in \zbb^c$. Then $p_{*}q^*D\in | \ls(\mathbf{b}) |$  by \eqref{eq:multiplicity}. 
\end{proof}
By similar arguments as \cref{claim:base control}, $ \ls(\mathbf{b})\boxtimes \oc_{\pb^n}(-1)_{\upharpoonright \ys}$ is effective, and its base locus
\begin{eqnarray}\label{locus debarre}
{\rm Bs}\big( \ls(\mathbf{b})\boxtimes  \oc_{\pb^n}(-1)_{\upharpoonright \ys}\big)\subset p^{-1}(G^{\infty}),
\end{eqnarray}
where $G^{\infty}$ is the set of points in $\mathbf{G}$ at which  the  fiber in $\ys$ is positive dimensional. 
  We can   apply the same methods in proving \cref{weak nakamaye} to obtain a more general result. 
\begin{thm}\label{effective nakamaye2}
	Let $\ys$ be the  universal complete intersection defined by 
	$$
	\ys:=\big\{(\Delta_1,\ldots,\Delta_c,[z])\in \prod_{i=1}^{c}{\rm Gr}_{k+1}(V_{\delta_i})\times \pb^n \mid  \Delta_i([z])=0, \forall\ i=1,\ldots,c\big\},
	$$
	where $V_{\delta_i}:=H^0\big(\pb^n,\oc_{\pb^n}(\delta_i)\big)$, and $(k+1)c\geqslant n$. For any $J\subset \{0,\ldots,n\}$,  define  $\ys_J:= \ys\cap \prod_{i=1}^{c}{\rm Gr}_{k+1}(V_{\delta_i})\times \pb_J $.   Then for  any $\af=(a_1,\ldots,a_c)\in \mathbb{N}^c$ with $a_i\geqslant \frac{\prod_{j=1}^{c}\delta_j^{k+1}}{\delta_i}$ for   $i=1,\ldots,c$,  the base locus
	$$
	{\rm Bs}\big( \ls(\mathbf{a})\boxtimes  \oc_{\pb^n}(-1)_{\upharpoonright \ys_J}\big)\subset p_J^{-1}(G_J^{\infty}),
	$$
	where $G_J^{\infty}$ is the set of points in $ \prod_{i=1}^{c}{\rm Gr}_{k+1}(V_{\delta_i})$ at which  the  fiber in $\ys_J$  is  positive dimensional.
\end{thm}

\begin{rem}
Very recently, Brotbek and the author  \cite{BD17,BD18} extended the techniques in \cite{BD15,Bro17} to the logarithmic settings using \emph{meromorphic connections}, and we proved
\begin{thmlist}
	\item the logarithmic analogue of the Debarre conjecture: for  general hypersurfaces $H_1,\ldots,H_n\in |\oc_{\pb^n}(d)|$ with $d\geqslant (4n)^n$ and $D:=\sum_{i=1}^{n}H_i$ simple normal crossing, the logarithmic cotangent bundle $\Omega_{\pb^n}(\log D)$ is \emph{almost ample}; 
	\item a result towards the \emph{orbifold Kobayashi conjecture} by Rousseau  \cite{Rou10}: for general hypersurfaces $H\in |\oc_{\pb^n}(d)|$ with $d\geqslant (n+1)^{n+3}\cdot (n+2)^{n+3}$, the \emph{Campana orbifold} $\big(\pb^n, (1-\frac{1}{d})H\big)$ is \emph{orbifold hyperbolic}. 
\end{thmlist}	
Let us mention that we have to apply \cref{effective nakamaye2}    to obtain the effective   lower bounds of degrees in \cite{BD17,BD18}.
\end{rem}

\section{On the Diverio-Trapani Conjecture}\label{proof of main}
In this section, we  apply  the techniques in \cite{BD15,Bro17} to prove \cref{main}.  Let us mention that \cref{sec:Fermat} is not self-contained, and we strongly recommend the readers who are interested in further details to refer to the paper \cite{Bro17}.  \noindent  
\subsection{Families of  Fermat-type Hypersurfaces}\label{sec:Fermat}
In \cite{Bro17}, Brotbek introduced the families of \emph{Fermat-type Hypersurfaces} as a  candidate for the examples satisfying a strong Zariski open property for hyperbolicity. In this subsection, we briefly recall his constructions  and the essential techniques  in \cite{Bro17} which will be used in the proof  of \cref{main}.

Let $X$ be an $n$-dimensional projective manifold endowed with a very ample line bundle $A$. 
 We fix $n+1$ sections in general position $\tau_0,\ldots,\tau_n\in H^0(X,A)$. 
Let us fix a positive integer $r$ and $k$.  For any $\ep,\delta\in \mathbb{N}$, set $V_{\delta}:=H^0\big(\pb^n,\oc_{\pb^n}(\delta)\big)$, and $\ab_{\ep,\delta}:=H^0(X,A^\ep)\otimes V_\delta$. Consider for any $\af:=\big(a_I\in H^0(X,A^\ep)\big)_{|I|=\delta}\in \ab_{\ep,\delta}$,   the hypersurface  $H_{\af}$ in $X$   defined by the zero locus of the  section 
\begin{eqnarray}\label{eq:Fermat section}
\sigma(\af):=\sum_{|I|=\delta}a_{I}\tau^{(r+k)I}\in H^0(X,A^{m})
\end{eqnarray}
where $m=\varepsilon+(r+k)\delta$ and $\tau^{(r+k)I}:=(\tau_0^{i_0}\cdots\tau_n^{i_n})^{r+k}$ for \(I=(i_0,\dots,i_n)\).  Consider the universal family 
$$
\hs_{\ep,\delta}:=\{(\af,x)\in  \ab_{\ep,\delta}\times X\mid \sigma(\af)(x)=0  \}.
$$
There exists  a Zariski open set of $\ab_{\ep,\delta}^{\rm sm}\subset\ab_{\ep,\delta}$ so that over $\ab_{\ep,\delta}^{\rm sm}$, $\hs_{\ep,\delta}$ is a smooth family. Let us also denote by $\hs_{\ep,\delta}\to \ab^{\rm sm}_{\ep,\delta}$ the restrict family, ${\hs}^{\rm rel}_{\ep,\delta,k}$ the (relative) Demailly-Semple $k$-jet  tower of $(\hs_{\ep,\delta},T_{\hs_{\ep,\delta}/\ab^{\rm sm}_{\ep,\delta}})$, and $\hat{\hs}^{\rm rel}_{\ep,\delta,k}$ the  blow-up of ${\hs}^{\rm rel}_{\ep,\delta,k}$ defined in \cref{fonctorial}.

Let us define a finite set $\Sigma:=\cup_{\{j_1,\ldots,j_n\}\subset \{0,\ldots,n\}}(\tau_{j_1}=\cdots=\tau_{j_n}=0) $ of $X$, and write $X^{\circ}:=X\setminus \Sigma$. Denote by $\hat{X}_k^\circ:=(\pi_{0,k}\circ \nu_k)^{-1}(X^\circ)$. We can shrink $\ab_{\ep,\delta}^{\rm sm}$ to a Zariski open set so that  $\hs_{\ep,\delta}\subset \ab_{\ep,\delta}^{\rm sm}\times  X^\circ$  and, \emph{a fortiori},  $\hat{\hs}^{\rm rel}_{\ep,\delta,k}\subset \ab_{\ep,\delta}^{\rm sm}\times  \hat{X}_k^\circ$. 

	We need to cover $X$ by a natural stratification induced by the vanishing of the $\tau_j$'s. For any $J\subset \{0,\ldots,n\}$, define 
\begin{align*}
X_J&:=\{x\in X \mid \tau_j(x)=0\Leftrightarrow j\in J   \},\\
\pb_J&:=\{[z]\in \pb^n\mid z_j=0 \mbox{ if }j\in J   \},\\
V_{\delta,J}&:=H^0\big(\pb_J,\oc_{\pb_J}(\delta) \big),
\end{align*}
$\hat{X}_{k,J}:=(\pi_{0,k}\circ\nu_k)^{-1}(X_J)$ and $\hat{X}_{k,J}^{\circ}:=\hat{X}_{k,J}\cap \hat{X}_k^{\circ}$. 
 
 We are now in position to recall the main results in \cite{Bro17}, which will be applied in \cref{sec:intersection}.
\begin{thm}[Brotbek]\label{thm:Brotbek}
Fix any $r\in \mathbb{N}$. For each $\ep,\delta\in \mathbb{N}$,  there exists a rational map
	\begin{align}\label{eq:rational map}
\Phi_{\ep,\delta}:  \ab_{\ep,\delta}\times  \hat{X}_k\dashrightarrow {\rm Gr}_{k+1}(V_{\delta})
	\end{align}
	induced by Brotbek's Wronskians.  Suppose that $\ep\geqslant m_{\infty}(X_k,A)$  and $\delta\geqslant n(k+1)$.
	\begin{thmlist}\item 
There exists a non-empty Zariski open subset $\ab_{\ep,\delta}^{\circ}\subset \ab^{\rm sm}_{\ep,\delta}$ so that the restriction of $\Phi_{\ep,\delta}$ to $\ab_{\ep,\delta}^{\circ}\times \hat{X}_k^{\circ}$ is a regular morphism.
	\item  	Set $\ls$ to be the tautological line bundle on $ {\rm Gr}_{k+1}(V_{\delta})$, and $F$ to be the effective divisor in $\hat{X}_k$ defined by $\oc_{\hat{X}_k}(-F):=\nu_k^*\wk_{\infty}(X_k)$. One has 
	\begin{align}
	\Phi_{\ep,\delta}^*\ls=\nu_k^*\big(\oc_{X_k}(k')\otimes \pi_{0,k}^*A^{(k+1)(\ep+k\delta)}\big)\otimes \oc_{\hat{X}_k}(-F).
	\end{align}

	\item Define a rational map 
	\begin{align*}
	\Psi_{\ep,\delta}:  \ab_{\ep,\delta}\times  \hat{X}_k &\dashrightarrow {\rm Gr}_{k+1}(V_{\delta})\times \pb^n\\
	(\af,w)  &\mapsto \big(\Phi_{\ep,\delta}(\af,w), [\tau^r(w)] \big),
	\end{align*}
	where $[\tau^r(w)]:=\big[\tau_0^r\big(\pi_{0,k}\circ \nu_k(w)\big),\ldots,\tau_n^r\big(\pi_{0,k}\circ \nu_k(w) \big)\big]$. The restriction of  $\Psi_{\ep,\delta}$   to $\hat{\hs}_{\ep,\delta,k}^{{\rm rel}}$ factors through $\ys$, where $\ys\subset {\rm Gr}_{k+1}(V_{\delta})\times \pb^n$ is the universal family of complete intersections defined in \eqref{universal}.  In other words, for any $\af\in \ab_{\ep,\delta}^{\circ}$, $\hat{H}_{\af,k}\subset \hat{X}_k^\circ$  and $\Psi_{\ep,\delta}(\hat{H}_{\af,k})\subset \ys$.  
	\item \label{eq:linear} For any $w\in \hat{X}_k^\circ$, there exists a $\cb$-linear map
	\begin{align}
	\varphi_{\ep,\delta,w}:\ab_{\ep,\delta}\to V_{\delta}^{k+1}
	\end{align}
	such that $\Phi_{\ep,\delta}$ is defined at $(\af,w)\in \ab_{\ep,\delta}\times \hat{X}_k^\circ$ if and only if 
	${\rm dim}\, [\varphi_{\ep,\delta,w}(\af)]=k+1$. Here $[\varphi_{\ep,\delta,w}(\af)]$  denotes to be the subspace in $V_\delta$ spanned by $(k+1)$-vectors $\varphi_{\ep,\delta,w}(\af)$.  Moreover, for any $\af\in \ab_{\ep,\delta}^\circ$, 
	$\Phi_{\ep,\delta}(\af,w)=   [\varphi_{\ep,\delta,w}(\af)]\in {\rm Gr}_{k+1}(V_\delta)$. 
	\item \label{rank} Same setting as  above.  For the (unique) $J\subset \{0,\ldots,n\}$ so that $w\in \hat{X}_{k,J}^\circ$,   the composition of $\cb$-linear maps
	$$
\phi_{\ep,\delta,w}:	\ab_{\ep,\delta}\xrightarrow{\varphi_{\ep,\delta,w}}V_{\delta}^{k+1}\xrightarrow{\rho_w} V_{\delta,J}^{k+1},
	$$
is surjective. Here $\rho_w:V_{\delta}^{k+1}\to   V_{\delta,J}^{k+1}$ is the projection map. 
	\end{thmlist}
	\end{thm}

\subsection{Families of complete intersections of Fermat-type hypersurfaces}\label{sec:intersection}
Let us  construct    families  of complete intersection varieties in $X$ cut out by  Fermat-type hypersurfaces defined in \cref{sec:Fermat}.  As we will see in \cref{k-jet}, these examples satisfy the strong Zariski open property \eqref{eq:star} for almost $k$-jet ampleness defined in \cref{almost k ample}.  

We fix $1\leqslant c\leqslant n-1$, $r\in \mathbb{N}$, $k\geqslant \frac{n}{c}-1$, and two $c$-tuples of positive integers $\bm{\ep}=(\ep_1,\ldots,\ep_c),\bm{\delta}=(\delta_1,\ldots,\delta_c)\in \mathbb{N}^c$.  
Consider the family $ {\mathscr{Z}}\subset \ab_{\ep_1,\delta_1}\times\cdots\times \ab_{\ep_c,\delta_c}\times X$ of \emph{complete intersection varieties} in $X$ defined by 
\begin{align}\label{family}
 {\mathscr{Z}}:=\{(\af_1,\ldots,\af_c,x)\in \ab_{\ep_1,\delta_1}\times\cdots\times \ab_{\ep_c,\delta_c}\times X\mid   \sigma(\af_1)(x)=\cdots=\sigma(\af_c)(x)=0 \},
\end{align}
where $\sigma(\af_i)$ is the   section defined in \eqref{eq:Fermat section}.  Let us denote by $\rho:\zs\to \ab_{\ep_1,\delta_1}\times\cdots\times \ab_{\ep_c,\delta_c}$ the natural projection, and for any $\af:=(\af_1,\ldots,\af_c)$, set $Z_\af:=\rho^{-1}(\af)$. 
One can show that  there is a non-empty  Zariski open set $\ab_{\rm sm}\subset \ab:= \ab_{\ep_1,\delta_1}\times\cdots\times \ab_{\ep_c,\delta_c}$ so that $Z_\af$ is smooth for any $\af\in \ab_{\rm sm}$. In other words, for any $\af\in \ab_{\rm sm}$, the $c$-hypersurfaces  $H_{\af_1},\ldots,H_{\af_c}$ are smooth and intersect transversely so that $Z_{\af}:=H_{\af_1}\cap \ldots\cap H_{\af_c}$ is a smooth subvariety in $X$ of codimension $c$.  Let us also denote by $\zs\to \ab_{\rm sm}$ the restricted (smooth) family.  Denote by $\zs^{\rm rel}_k$ the relative  Demailly-Semple $k$-jet tower of $(\zs,T_{\zs/\ab_{\rm sm}})$, and $\hat{\zs}^{\rm rel}_k$ its blow-up defined in \cref{fonctorial}. Observe that $Z_{\af,k}=H_{\af_1,k}\cap\ldots\cap H_{\af_c,k}$ for any $\af\in \ab_{\rm sm}$, and by   \cref{fonctorial}, one has
\begin{align}\label{eq:intersection def}
\hat{Z}_{\af,k}\subset \hat{H}_{\af_1,k}\cap\ldots\cap \hat{H}_{\af_c,k}.
\end{align}  

Consider a rational map $\Phi:\ab\times \hat{X}_k\dashrightarrow {\rm Gr}_{k+1}(V_{\delta_1})\times \cdots\times  {\rm Gr}_{k+1}(V_{\delta_c})$ by taking the products of  \eqref{eq:rational map}. Precisely speaking, $\Phi$ is defined by
\begin{align*}
\Phi:\ab\times \hat{X}_k &\dashrightarrow {\rm Gr}_{k+1}(V_{\delta_1})\times \cdots\times  {\rm Gr}_{k+1}(V_{\delta_c})\\
(\af_1,\ldots,\af_c,w) & \mapsto \big(\Phi_{\ep_1,\delta_1}(\af_1,w),\ldots, \Phi_{\ep_c,\delta_c}(\af_c,w)  \big)
\end{align*}
Write $\gf:={\rm Gr}_{k+1}(V_{\delta_1})\times \cdots\times  {\rm Gr}_{k+1}(V_{\delta_c})$ for short. As a direct consequence of \cref{wronskian result,thm:Brotbek}, we have the following result.
\begin{thm}\label{thm:variant} Assume that   $\ep_i\geqslant k,\delta_i\geqslant n(k+1)$ for every $i=1,\ldots,c$. Then
		\begin{thmlist}
			\item     the restriction of $\Phi$ to $\ab^\circ_{\ep_1,\delta_1}\times \cdots\times \ab^\circ_{\ep_c,\delta_c}\times \hat{X}_k^\circ$ is regular. 
			\item Set $\ab^\circ:=\ab^\circ_{\ep_1,\delta_1}\times \cdots\times \ab^\circ_{\ep_c,\delta_c}\cap \ab_{\rm sm}$. We also denote by   $\hat{\zs}_k^{\rm rel}\to \ab^\circ$ the restricted family. Then $\hat{\zs}_k^{\rm rel}\subset \ab^\circ\times \hat{X}_k^\circ$.
		  	\item  For any $ (b_1,\ldots,b_c)\in \mathbb{N}^c$,   one has
			\begin{align}\label{eq:pull-back}
			\Phi^*\ls(b_1,\ldots,b_c)=\nu_k^*\big(\oc_{X_k} (\sum_{i=1}^{c}b_ik')\otimes \pi_{0,k}^*A^{\sum_{i=1}^{c}b_i(k+1) (\ep_i+k\delta_i )}\big)\otimes \oc_{\hat{X}_k}\big(-(\sum_{i=1}^{c}b_i)F\big),
			\end{align}
			where $\ls(b_1,\ldots,b_c)$ is the tautological line bundle defined in \eqref{eq:group}.
 		\item \label{factor through}Define a rational map 
		\begin{align*}
		\Psi:  \ab \times  \hat{X}_k &\dashrightarrow {\rm Gr}_{k+1}(V_{\delta_1})\times \cdots\times  {\rm Gr}_{k+1}(V_{\delta_c})\times \pb^n\\
		(\af,w)  &\mapsto \big(\Phi(\af,w), [\tau^r(w)] \big),
		\end{align*}
		where $[\tau^r(w)]:=\big[\tau_0^r\big(\pi_{0,k}\circ \nu_k(w)\big),\ldots,\tau_n^r\big(\pi_{0,k}\circ \nu_k(w) \big)\big]$. The restriction of  $\Psi$   to $\hat{\zs}_{k}^{{\rm rel}}$ factors through $\ys$, where $\ys\subset {\rm Gr}_{k+1}(V_{\delta_1})\times \ldots\times  {\rm Gr}_{k+1}(V_{\delta_c})\times \pb^n$ is the universal family of complete intersections  defined in \eqref{eq:generalized universal}.  In other words, for any $\af\in \ab^{\circ}$, $\hat{Z}_{\af,k}\subset \hat{X}_k^\circ$  and $\Psi (\hat{Z}_{\af,k})\subset \ys$.  
	\end{thmlist}
\end{thm}
\begin{proof}
We apply \cref{wronskian result} to take $m_{\infty}(X_k,A)=k$.   (\lowerromannumeral{1}), (\lowerromannumeral{2}) and (\lowerromannumeral{3}) can be easily derived from \cref{thm:Brotbek}.  To prove (\lowerromannumeral{4}), it is enough to show that for any $\af\in \ab^{\circ}$,  $\Psi (\hat{Z}_{\af,k})\subset \ys$. By \eqref{eq:intersection def}, for any $w\in \hat{Z}_{\af,k} $, $i=1,\ldots,c$ and   $P\in \Phi_{\ep_i,\delta_i}(\af_i,w)$, one has
$$
P\big([\tau^r(w)]\big)=0.
$$
This proves  (\lowerromannumeral{4}) by the definition of $\ys$. 
\end{proof}

 Set $\ys_J:=\ys\cap\big(\gf\times\pb_J\big)\subset \gf\times\pb^n$, and denote by $G^\infty_J$    the set of points in $\mathbf{G}$ at which the fiber in $\ys_J$ is positive dimensional.

Now we are ready to prove the following   lemma, which is a variant of \cite[Lemma 3.11]{Bro17}.
\begin{lem}[Avoiding positive dimensional fibers]\label{exceptional locus}
 Assume that $\ep_i\geqslant k$, $\delta_i\geqslant {\rm dim}\, \hat{X}_k=(n-1)(k+1)+1$ for  $i=1,\ldots,c$.   Then 	for any $J\subset \{0,\ldots,n\}$, there exists a non-empty Zariski open subset $\ab_J\subset \ab^{\circ}$ such that
	$$
	{\Phi}^{-1}(G_J^\infty)\cap (\ab_J\times \hat{X}_{k,J}^{\circ} )=\varnothing .
	$$
\end{lem}
\begin{proof}
We introduce the following analogues of $\ys_J$ parametrized by affine spaces
\begin{align*} 	 \tl{\ys}_{1,J}:=\big\{\big(\alpha_{10},\ldots, \alpha_{ck},[z]\big)\in \prod_{i=1}^{c}V_{\delta_i}^{k+1}\times \pb_{J}\mid \forall 1\leqslant i\leqslant c,0\leqslant p\leqslant k, \alpha_{ip}([z])=0 \big\}, \\
 	\tl{\ys}_{2,J}:=\{\big(\alpha_{10}, \ldots,\alpha_{ck},[z]\big)\in \prod_{i=1}^{c}V_{\delta_i,J}^{k+1}\times \pb_{J}\mid \forall 1\leqslant i\leqslant c,0\leqslant p\leqslant k, \alpha_{ip}([z])=0 \}.
	\end{align*}
By analogy with $G_J^\infty$, we denote by $\mathbb{V}_{1,J}^\infty$ (resp. $\mathbb{V}_{2,J}^\infty$) the set of points in $\prod_{i=1}^{c} V_{\delta_i} ^{k+1}$ (resp. $\prod_{i=1}^{c} V_{\delta_i,J} ^{k+1}$) at which the fiber in $\tl{\ys_{1,J}}$ (resp. $\tl{\ys}_{2,J}$) is positive dimensional.
	
	Fix any $w\in \hat{X}_{k,J}^{\circ}$. 	By  \cref{eq:linear},   for any $\af=(\af_1,\ldots,\af_c)\in \ab^{\circ}$ we have
	$$
	{\Phi}(\af,{w})=\big([{\varphi}_{\ep_1,\delta_1,w}(\af_1)],\ldots,[{\varphi}_{\ep_c,\delta_c,w}(\af_c)] \big),
	$$
	where   ${\varphi}_{\ep_i,\delta_i,w}:\ab_{\ep_i,\delta_i}\to V_{\delta_i}^{k+1} $ is the linear map defined in \cref{eq:linear}. Let us define a $\cb$-linear map
	\begin{align*}
	\varphi_w: \ab &\to \prod_{i=1}^{c}V_{\delta_i}^{k+1}\\
	\af &\mapsto   \big({\varphi}_{\ep_1,\delta_1,w}(\af_1),\ldots,{\varphi}_{\ep_c,\delta_c,w}(\af_c)\big).
	\end{align*}  Then we have
	$$
	{\Phi}^{-1}(G^\infty_J)\cap (\ab^{\circ}\times \{{w}\})= {\varphi}_{{w}}^{-1}(\mathbb{V}_{1,J}^\infty)\cap \ab^{\circ}=(\rho_w\circ {\varphi}_{{w}})^{-1}(\mathbb{V}_{2,J}^\infty)\cap \ab^{\circ},
	$$
	where 
	\begin{eqnarray}\nonumber
	\rho_w:\prod_{i=1}^{c}V_{\delta_i}^{k+1}\rightarrow \prod_{i=1}^{c}V_{\delta_i,J}^{k+1}
	\end{eqnarray}
	is the  projection map. 
	Since the linear map $\rho_w\circ {\varphi}_{{w}}$ is diagonal by blocks, by  \cref{rank} we have
	$$
	{\rm rank}\rho_w\circ{\varphi}_{{w}}=\sum_{i=1}^{c}(k+1)\dim V_{\delta_i,J}.
	$$
	Therefore
	\begin{align*}
		{\rm dim}\big({\Phi}^{-1}(G_J^\infty)\cap (\ab^{\circ}\times \{{w}\})\big)&\leqslant  {\rm dim}\big((\rho_w\circ {\varphi}_{{w}})^{-1}(\mathbb{V}_{2,J}^\infty)\big)\\
		&\leqslant  {\rm dim}(\mathbb{V}_{2,J}^\infty)+{\rm dim\,  ker}(\rho_w\circ{\varphi}_{{w}})\\
		&\leqslant {\rm dim}(\mathbb{V}_{2,J}^\infty)+{\rm dim}\, \ab-{\rm rank}(\rho_w\circ{\varphi}_{{w}})\\
		&= {\rm dim}(\mathbb{V}_{2,J}^\infty)+{\rm dim}\, \ab-\sum_{i=1}^{c}(k+1)\dim V_{\delta_i,J}\\
		&= {\rm dim}\, \ab-{\rm codim}(\mathbb{V}_{2,J}^\infty,\prod_{i=1}^{c}V_{\delta_i,J}^{k+1}),
	\end{align*}
	which in turn implies that
	$$
	{\rm dim}\big({\Phi}^{-1}(G_J^\infty)\cap \ab^{\circ}\times \hat{X}_{k,J}^{\circ}\big)\leqslant {\rm dim}\, \ab-{\rm codim}(\mathbb{V}_{2,J}^\infty,\prod_{i=1}^{c}V_{\delta_i,J}^{k+1})+{\rm dim}\, \hat{X}_k.
	$$
	By a result due to   Benoist \cite{Ben11} and Brotbek-Darondeau (see \cite[Corollary 3.2]{BD15}), we have
	$$
	{\rm codim}(\mathbb{V}_{2,J}^\infty,\prod_{i=1}^{c}V_{\delta_i,J}^{k+1}) \geqslant \min \limits_{i=1,\ldots,c}\delta_i+1.
	$$
	Therefore, if 
	\begin{equation}\label{avoid} 
 	{\rm dim}\, \hat{X}_k<\min \limits_{i=1,\ldots,c}\delta_i+1,
	\end{equation}
	${\Phi}^{-1}(G_J^\infty)$ doesn't dominate $\ab^{\circ}$ via the projection $\ab^{\circ}\times\hat{X}_{k,J}^{\circ}\rightarrow \ab^{\circ}$, and  we can thus find a non-empty Zariski open subset $\ab_J\subset \ab^{\circ}$ such that
	$$
	{\Phi}^{-1}(G_J^\infty)\cap (\ab_J\times \hat{X}_{k,J}^{\circ} )=\varnothing .
	$$
\end{proof}

\subsection{Proof of \cref{main}}\label{main proof}
We are now in position to prove \cref{main}.  Indeed, we establish the  following more refined result than \cref{main}. 
\begin{thm}\label{k-jet}
	Let $X$ be an $n$-dimensional  projective manifold equipped with a very ample line bundle $A$. Let $c$ be any integer satisfying $1\leqslant  c\leqslant n-1$, and set $k:=\lceil \frac{n}{c}\rceil -1$.  Assume that the multi-degrees $(d_1,\ldots,d_c)\in (\mathbb{N})^{c}$ satisfy the following condition:
	\begin{align*} 
	&\exists \bm{\delta}:=(\delta_1,\ldots,\delta_c)\in \mathbb{N}^c\  {\rm with} \ \delta_i\geqslant\delta_0:= n(k+1) \ {\rm for}\ i=1,\ldots,c. \\ 
	&\exists \bm{\ep}:=(\ep_1,\ldots,\ep_c)\in \mathbb{N}^c\  {\rm with} \ \ep_i\geqslant k  \ {\rm for}\ i=1,\ldots,c.\\   
	&\exists r>\sum_{i=1}^{c}b_i(k+1)(\ep_i+k\delta_i), \ {\rm where}\ b_i:= \frac{\prod_{j=1}^{c}\delta_j^{k+1}}{\delta_i} \\ 
   &{\rm s.t.}\ \ 	d_i= \ep_i +(r+k) \delta_i \ {\rm for}\ i=1,\ldots,c.
	\end{align*}
Then  for general hypersurfaces $H_1\in  |A^{d_1}|,\ldots,H_c\in |A^{d_c}|$, 	their complete intersection (smooth) variety $Z:=H_1\cap \ldots \cap H_c$ is  \emph{almost $\tilde{k}$-jet ample} for   any $\tilde{k}\geqslant k$.
\end{thm}
\begin{proof}
Observe that, the choice for $(\bm{\ep},\bm{\delta})$ and $k$ in the theorem fits  all the requirements in   \cref{thm:variant,exceptional locus}. 
In the same vein as  \cite{BD15,Bro17}, let us first prove  the nefness.
\begin{claim}\label{nef claim}
Set $\ab_{\rm nef}:=\cap_{J}\ab_J$. For any $\af\in \ab_{\rm nef}$, the line bundle 
$$
\nu_k^*\big(\oc_{X_k}(\sum_{i=1}^{c}b_ik')\otimes \pi_{0,k}^*A^{-q(\bm{\ep},\bm{\delta},r)}\big)\otimes \oc_{\hat{X}_k}(-\sum_{i=1}^{c}b_iF)_{\upharpoonright \hat{Z}_{\af,k}}
$$
on $\hat{Z}_{\af,k} $ is nef. Here we write  $q(\bm{\ep},\bm{\delta},r):=r-\sum_{i=1}^{c}b_i(k+1)(\ep_i+k\delta_i)>0$.
\end{claim}
\begin{proof}
To prove that a line bundle on a projective variety is nef, it suffices to show that for any irreducible curve, its intersection with this line bundle is non-negative. For any fixed $\af\in  \ab_{\rm nef}$, and any irreducible curve $C\subset \hat{Z}_{\af,k}$, there is a unique $J\subset \{0,\ldots,n\}$ such that $C^{\circ}:=\hat{X}^\circ_{k,J}\cap C$ is a non-empty Zariski open subset of $C$, and thus $C^{\circ}\subset \hat{\zs}_{k,J}$.  It follows from   \cref{factor through} that   $\Psi$ factors through $\ys_J$ when restricted to $\hat{\zs}_{k,J}$. Hence $\Psi_{\upharpoonright C^{\circ}}$ also factors through $\ys_J$, and by the properness of $\ys_J$, $\Psi(C)\subset \ys_J$.  By  \cref{exceptional locus} and the definition of  $ \ab_{\rm nef}$, we have
	$$
	{\Phi}(C^{\circ})\cap G_J^{\infty}=\varnothing ,
	$$
	and thus 
	$$
	\Psi(C)\not\subset p_J^{-1}(G_J^{\infty}).
	$$
By \cref{effective nakamaye2}, one has
$$
	{\rm Bs}\big(\ls(b_1,\ldots,b_c)\boxtimes  \oc_{\pb^n}(-1)_{\upharpoonright \ys_J}\big)\subset p_J^{-1}(G_J^{\infty}),
$$
which yields
$$
\Psi(C)\cdot \big(\ls(b_1,\ldots,b_c)\boxtimes  \oc_{\pb^n}(-1)_{\upharpoonright \ys}\big)=\Psi(C)\cdot \big(\ls(b_1,\ldots,b_c)\boxtimes  \oc_{\pb^n}(-1)_{\upharpoonright \ys_J}\big)\geqslant 0.
$$
Write $\Psi_{\af}:\hat{Z}_{\af,k}\to \ys$ the restriction of $\Psi$ to $\hat{Z}_{\af,k}$. 
By  \eqref{eq:pull-back}, we have
$$
\Psi_\af^*\big(\ls(b_1,\ldots,b_c)\boxtimes  \oc_{\pb^n}(-1)_{\upharpoonright \ys}\big)= \nu_k^*\big(\oc_{X_k}(\sum_{i=1}^{c}b_ik')\otimes \pi_{0,k}^*A^{-q(\bm{\ep},\bm{\delta},r)}\big)\otimes \oc_{\hat{X}_k}(-\sum_{i=1}^{c}b_iF)_{\upharpoonright \hat{Z}_{\af,k} },
$$
and thus 
$$
C\cdot \Big(\nu_k^*\big(\oc_{X_k}(\sum_{i=1}^{c}b_ik')\otimes \pi_{0,k}^*A^{-q(\bm{\ep},\bm{\delta},r)}\big)\otimes \oc_{\hat{X}_k}(-\sum_{i=1}^{c}b_iF)_{\upharpoonright \hat{Z}_{\af,k} }\Big)\geqslant 0,
$$
which proves the claim.
\end{proof}	
\medskip

By   \cite[Proposition 6.16]{Dem95}, we can find an ample line bundle 
	$$
 	\nu_k^*\big(\oc_{X_k}(a_k,\ldots,a_1)\otimes \pi_{0,k}^*A^{a_0}\big)\otimes \oc_{\hat{X}_k}(-F)
	$$
on $\hat{X}_k$  	for some $a_0,\ldots,a_k\in \mathbb{N}$.  Denote by $\nu_{\af,k}:\hat{Z}_{\af,k}\to Z_{\af,k}$   the blow-up of the asymptotic Wronskian ideal sheaf $\wk_{\infty}(Z_{\af,k})$ of $Z_{\af,k}$. Write $A_{\af}:=A_{\upharpoonright Z_{\af}}$ and  $F_{\af}:=F\cap \hat{Z}_{\af,k}$. Therefore, for any $\ell>a_0$, by \cref{nef claim}
the line bundle
\begin{align*}
\nu_{\af,k}^*\big(\oc_{Z_{\af,k}}(a_k+\sum_{i=1}^{c}\ell b_ik',a_{k-1},\ldots,a_{1})\otimes \pi_{0,k}^*A_{\af}^{a_0-\ell q(\bm{\ep},\bm{\delta},r)}\big)\otimes \oc_{\hat{Z}_{\af,k}}\big(-(\sum_{i=1}^{c}\ell b_i+1)F_{\af} \big)=\\
\nu_k^*\big(\oc_{X_k}(a_k+\sum_{i=1}^{c}\ell b_ik',a_{k-1},\ldots,a_{1})\otimes \pi_{0,k}^*A^{a_0-\ell q(\bm{\ep},\bm{\delta},r)}\big)\otimes \oc_{\hat{X}_{k}}\big(-(\sum_{i=1}^{c}\ell b_i+1)F \big)_{\upharpoonright \hat{Z}_{\af,k}}
\end{align*}
is ample for  $\af\in  \ab_{\rm nef}$, which verifies the condition \eqref{eq:star}.    
  By the Zariski open property  \eqref{eq:star}  in \cref{Zariski}, we conclude that there exists a non-empty Zariski open subset $S_{\rm ample}\subset \prod_{i=1}^{c}|A^{d_i}|$ such that for any $(H_1,\ldots,H_c)\in S_{\rm ample}$, their complete intersection $Z:=H_1\cap\ldots\cap H_c$ is a reduced smooth variety of codimension $c$ in $X$, and $Z$ is almost $k$-jet ample. 
By  \cite[Lemma 7.6]{Dem95}, if a complex manifold $Y$ is almost $k$-jet ample, then it is also  almost $\tilde{k}$-jet ample   for any $\tilde{k}\geqslant k$. This finishes the proof of the theorem.
\end{proof}

Let us deduce  \cref{main} from \cref{k-jet}.
\begin{proof}[Proof of \cref{main}]
	Let us keep the same notations in \cref{k-jet}. We will fix  $\ep_1=\cdots=\ep_c\geqslant k:=\lceil \frac{n}{c}\rceil -1$ and $\bm{\delta}:=(\delta_0,\ldots,\delta_0)$ with $\delta_0=n(k+1)$. Then $b_1=\cdots=b_c=\delta_0^{c(k+1)-1}$.  If we take $$d_0:=\delta_0\big(c(k+1)(k+\delta_0+k\delta_0-1)\delta_0^{c(k+1)-1}+1+k\big)+k,$$
then any $d\geqslant d_0$ has a decomposition
$$
d=\delta_0(r+k)+\ep
$$
with $k\leqslant\ep<k+\delta_0$, and $$r\geqslant c\delta_0^{c(k+1)-1}(k+1)(k+\delta_0-1+k\delta_0)+1>\sum_{i=1}^{c}b_i(k+1)(\ep+k\delta_0),$$ satisfying the conditions in  \cref{k-jet}. Observe that
\begin{align}\label{eq:lower bound}
d_0&=\delta_0\big(c(k+1)(k+\delta_0+k\delta_0-1)\delta_0^{c(k+1)-1}+1+k\big)+k\\\nonumber
&\leqslant \delta_0^{c(k+1)}c(k+1)^2(\delta_0+1) \\ \nonumber
&\leqslant 2c n^{c\lceil \frac{n}{c}\rceil+1}\cdot \lceil \frac{n}{c}\rceil^{c\lceil \frac{n}{c}\rceil+3} 
\end{align}
In conclusion,  the complete intersection $H_1\cap \ldots \cap H_c$ of $c$-general hypersurfaces $H_1,\ldots,H_c\in |\as^{d}|$ with $d\geqslant2c n^{c\lceil \frac{n}{c}\rceil+1}\cdot \lceil \frac{n}{c}\rceil^{c\lceil \frac{n}{c}\rceil+3} $ is  almost $\tilde{k}$-jet ample for any $\tilde{k}\geqslant  \frac{n}{c} -1$.
\end{proof}

Let us mention that when $\frac{n}{2}\leqslant c\leqslant n-1$, 	by \cite[Corollary 2.9]{BD15}, one can take $\delta_0:=2n-1$, which is slightly better than that in \cref{k-jet}. Now we apply the estimate in \cite{BD15} to provide a slight better bound in the case $\frac{n}{2}\leqslant c\leqslant n-1$.
\begin{proof}[Proof of \cref{main:Debarre}]
  Note that if $X$ is a smooth projective variety whose cotangent bundle $\Omega_X$ is ample, then for any smooth closed  subvariety  $Y\subset X$, $\Omega_Y$ is also ample. Hence it suffice to prove the theorem for $c=\lceil\frac{n}{2}\rceil$, $k=1$. By \eqref{eq:lower bound} and $\delta_0=2n-1$, one can take
    \begin{align*} 
 d_{{\rm Deb},n}&=\delta_0\big(c(k+1)(k+\delta_0+k\delta_0-1)\delta_0^{c(k+1)-1}+1+k\big)+k\\\nonumber
 & = 4(2n-1)^{2\lceil \frac{n}{2}\rceil+1} \cdot \lceil \frac{n}{2}\rceil+2(2n-1)+1\\
  & \leqslant  2(2n-1)^{n+2} \cdot (n+1)+4n-1\\
 &\leqslant (2n)^{n+3}.
 \end{align*}
\end{proof}

\subsection{Proof of \cref{main 3}}
This subsection is devoted to prove  \cref{main 3}.  
\begin{proof}[Proof of \cref{main 3}]

  Recall that  the  Demailly-Semple $k$-jet tower $Z_k$ of $(Z,T_Z)$  is a locally trivial product as well as its blow-up $\nu_k:\hat{Z}_k\to Z_k$ along the Wronskian ideal sheaf $\wk_{\infty}(Z_k)$. 
 Indeed, by \cref{brotbek wronskian}  for any $z\in Z$  there exists an open set $U$ containing $z$ so that $U_k:=\pi_{0,k}^{-1}(U)\simeq U\times \mathbb{R}_{n-c,k}$ and $\wk_{\infty}(Z_k)_{\upharpoonright U_k}\simeq {\rm pr}_2^*\mathfrak{I}_{n-c,k}$, where ${\rm pr}_2:U\times \mathbb{R}_{n-c,k}\to \mathbb{R}_{n-c,k}$ is the projection map.  Let us denote by $\mu_k:\hat{\rc}_{n-c,k}\to \rc_{n-c,k}$ the blow-up of $\rc_{n-c,k}$ along $\mathfrak{I}_{n-c,k}$. Write $\hat{U}_k:=\nu_k^{-1}(U_k)$.  Then
 \begin{align}\label{eq:local trivial}
  \xymatrix{
 	\hat{U}_k \ar[r]^-{\simeq} \ar[d]_-{\nu_k} & U\times \hat{\rc}_{n-c,k} \ar[d]^-{\mathds{1}\times \mu_k}\\
 	U_k \ar[r]^-{\simeq} \ar[r] & U\times \rc_{n-c,k} \\
 }
 \end{align}
  It follows from the proof of \cref{k-jet} that, there exists $a_1,\ldots,a_k,q\in \mathbb{N}$ such that 
  $\nu_k^*\oc_{Z_k}(a_k,\ldots,a_1)\otimes \oc_{\hat{Z}_k}(-qF)$ is ample. Write $\pi_k=\pi_{0,k}\circ \nu_k:\hat{Z}_k\to Z$. One thus can take $a_1,\ldots,a_k,q\gg 0$ so that all higher direct images 
  \begin{align}\label{vanishing}
  R^i(\pi_{k})_*\big(\nu_k^*\oc_{Z_k}(a_k,\ldots,a_1)\otimes \oc_{\hat{Z}_k}(-qF)\big)=0 \quad \forall \ i>0,
  \end{align}   
  and $\ls:=\nu_k^*\oc_{Z_k}(a_k,\ldots,a_1)\otimes \oc_{\hat{Z}_k}(-qF)\otimes\pi_k^*\as^{-1}$ is  ample   for some very ample line bundle $\as$ on $Z$.
  \begin{claim}\label{direct image ample}
$(\pi_{k})_*\big(\nu_k^*\oc_{Z_k}(ma_k,\ldots,ma_1)\otimes \oc_{\hat{Z}_k}(-mqF)\big)$ is an ample vector bundle for each $m\gg 0$.
  	\end{claim}
  \begin{proof}[Proof of \cref{direct image ample}]
 Denote by $\mathscr{E}_m:=(\pi_k)_*(\ls^m)$. 	From the   local trivial product structure of $\hat{Z}_k$ as in \eqref{eq:local trivial}, $\es_m$ is locally free for each $m>0$.

  	By \eqref{vanishing} and the degeneration of Leray spectral sequences, one has
  		\[
  	H^i(Z, \es_m\otimes \fc )=H^i(\hat{Z}_k,\ls^m\otimes \pi_k^*\fc) \quad \forall \ i>0,\ m>0
  	\]
  	for any coherent sheaf $\fc$ on $Z$. 
  	Fix any  point $y\in Z$, with the maximal ideal of $\oc_{Z,y}$ denoted by $\mc_y$. 
  	As $\ls$ is  ample, there is a positive integer $m_y\gg 0$ such that
  	$$
  	H^1(Z,\es_m\otimes \mc_y )=H^1(\hat{Z}_k,\ls^m\otimes \pi_k^*\mc_y)=0 \quad \forall \ m\geqslant m_y,
  	$$
   which in turn implies that $\es_m$ is globally generated  at $y$ for all $m\geqslant m_y$. 
   As $Z$ is compact,  we can find an integer $m_0\gg 0$ such that  $\es_m$ is globally generated when $m\geqslant m_0$. Observe that
   $$
  \es_m=(\pi_{k})_*\big(\nu_k^*\oc_{Z_k}(ma_k,\ldots,ma_1)\otimes \oc_{\hat{Z}_k}(-mqF)\big)\otimes \as^{-m}
   $$
   where $\as$ is a very ample line bundle on $Z$. Hence $(\pi_{k})_*\big(\nu_k^*\oc_{Z_k}(ma_k,\ldots,ma_1)\otimes \oc_{\hat{Z}_k}(-mqF)\big)$ is a quotient of a direct sum of copies of the very ample line bundle  $\oc_Z(\as^m)$.  	By the \emph{cohomological characterization of ample vector bundles} in \cite[Theorem 6.1.10]{Laz04II},  $(\pi_{k})_*\big(\nu_k^*\oc_{Z_k}(ma_k,\ldots,ma_1)\otimes \oc_{\hat{Z}_k}(-mqF)\big)$ is ample for $m\geqslant m_0$.
  \end{proof}   
By the projection formula
\begin{align}\label{eq:direct image}
\mathscr{F}:=(\pi_{k})_*\big(\nu_k^*\oc_{Z_k}(a_k,\ldots,a_1)\otimes \oc_{\hat{Z}_k}(-qF)\big)=(\pi_{0,k})_*\big(\oc_{Z_k}(a_k,\ldots,a_1)\otimes \mathfrak{J}_{q}\big),
\end{align}
where $\mathfrak{J}_{q}:=(\nu_{k})_*\oc_{\hat{Z}_k}(-qF)$ is the ideal sheaf of $ {Z_k}$ with the subscheme $\oc_{Z_k}/\mathfrak{J}_{q}$ supported on $Z_k^{\rm sing}$. By \cref{direct image ample}, for proper $a_1,\ldots,a_k,q\gg 0$, $\nu_k^*\oc_{Z_k}(a_k,\ldots,a_1)\otimes \oc_{\hat{Z}_k}(-qF)\otimes \pi_k^*\as^{-1}$ is very ample. For any regular germ of curve $f:(\cb,0)\to (Z,z)$, its $k$-th lift $f_{[k]}\in Z_k^{\rm reg}$.   Hence there exists a global section $\sigma\in H^0\big(Z_k,\oc_{Z_k}(a_k,\ldots,a_1)\otimes \pi_{0,k}^*\as^{-1} \otimes \mathfrak{J}_{q}\big)$ so that $\sigma(f_{[k]})\neq 0$.  Let $P_\sigma\in H^0(Z,\mathscr{F}\otimes \as^{-1})$ be the corresponding element of $\sigma$ under the isomorphism \eqref{eq:direct image}. Hence $P_{\sigma}([f]_k)\neq 0$. It follows from \cite[Proposition 6.16.\lowerromannumeral{1}]{Dem95} that  $
\mathscr{F}\subset E_{k,m}T^*_Z
$ for $m:=a_1+\cdots+a_k$. The corollary is thus proved. 
\end{proof}

\bibliographystyle{smfalpha}
\bibliography{biblio}

\providecommand{\bysame}{\leavevmode ---\ }
\providecommand{\og}{``}
\providecommand{\fg}{''}
\providecommand{\smfandname}{\&}
\providecommand{\smfedsname}{\'eds.}
\providecommand{\smfedname}{\'ed.}
\providecommand{\smfmastersthesisname}{M\'emoire}
\providecommand{\smfphdthesisname}{Th\`ese}
\begin{thebibliography}{{Dem}18}

\bibitem[BDa17]{BD15}
{\scshape D.~Brotbek {\normalfont \smfandname} L.~Darondeau} -- {\og Complete
  intersection varieties with ample cotangent bundles\fg}, \emph{Inventiones
  mathematicae} (2017).

\bibitem[BD17]{BD17}
{\scshape D.~{Brotbek} {\normalfont \smfandname} Y.~{Deng}} -- {\og {On the
  positivity of the logarithmic cotangent bundle}\fg}, \emph{ArXiv e-prints}
  (2017).

\bibitem[BD18]{BD18}
\bysame , {\og {Hyperbolicity of the complements of general hypersurfaces of
  high degree}\fg}, \emph{ArXiv e-prints} (2018).

\bibitem[Ben11]{Ben11}
{\scshape O.~Benoist} -- {\og Le th\'eor\`eme de {B}ertini en famille\fg},
  \emph{Bulletin de la Soci\'et\'e Math\'ematique de France} \textbf{139}
  (2011), no.~4, p.~555--569.

\bibitem[Bro16]{Bro16}
{\scshape D.~Brotbek} -- {\og Symmetric differential forms on complete
  intersection varieties and applications\fg}, \emph{Math. Ann.} \textbf{366}
  (2016), no.~1-2, p.~417--446.

\bibitem[Bro17]{Bro17}
\bysame , {\og On the hyperbolicity of general hypersurfaces\fg}, \emph{Publ.
  Math. Inst. Hautes \'Etudes Sci.} \textbf{126} (2017), p.~1--34.

\bibitem[Dar16]{Dar15}
{\scshape L.~Darondeau} -- {\og On the logarithmic {G}reen-{G}riffiths
  conjecture\fg}, \emph{Int. Math. Res. Not. IMRN} (2016), no.~6,
  p.~1871--1923.

\bibitem[Deb05]{Deb05}
{\scshape O.~Debarre} -- {\og Varieties with ample cotangent bundle\fg},
  \emph{Compositio Mathematica} \textbf{141} (2005), no.~6, p.~1445--1459.

\bibitem[DEG97]{DE97}
{\scshape J.-P. Demailly {\normalfont \smfandname} J.~El~Goul} -- {\og
  Connexions m\'eromorphes projectives partielles et vari\'et\'es alg\'ebriques
  hyperboliques\fg}, \emph{C. R. Acad. Sci. Paris S\'er. I Math.} \textbf{324}
  (1997), no.~12, p.~1385--1390.

\bibitem[Dem97]{Dem95}
{\scshape J.-P. Demailly} -- {\og Algebraic criteria for {K}obayashi hyperbolic
  projective varieties and jet differentials\fg}, in \emph{Algebraic
  geometry---{S}anta {C}ruz 1995}, Proc. Sympos. Pure Math., vol.~62, Amer.
  Math. Soc., Providence, RI, 1997, p.~285--360.

\bibitem[Dem11]{Dem11}
\bysame , {\og Holomorphic {M}orse inequalities and the
  {G}reen-{G}riffiths-{L}ang conjecture\fg}, \emph{Pure Appl. Math. Q.}
  \textbf{7} (2011), no.~4, Special Issue: In memory of Eckart Viehweg,
  p.~1165--1207.

\bibitem[{Dem}18]{Dem18}
{\scshape J.-P. {Demailly}} -- {\og {Recent results on the Kobayashi and
  Green-Griffiths-Lang conjectures}\fg}, \emph{ArXiv e-prints} (2018).

\bibitem[Div08]{Div08}
{\scshape S.~Diverio} -- {\og Differential equations on complex projective
  hypersurfaces of low dimension\fg}, \emph{Compos. Math.} \textbf{144} (2008),
  no.~4, p.~920--932.

\bibitem[DMR10]{DMR10}
{\scshape S.~Diverio, J.~Merker {\normalfont \smfandname} E.~Rousseau} -- {\og
  Effective algebraic degeneracy\fg}, \emph{Invent. Math.} \textbf{180} (2010),
  no.~1, p.~161--223.

\bibitem[DR15]{DR15}
{\scshape S.~Diverio {\normalfont \smfandname} E.~Rousseau} -- {\og The
  exceptional set and the {G}reen-{G}riffiths locus do not always coincide\fg},
  \emph{Enseign. Math.} \textbf{61} (2015), no.~3-4, p.~417--452.

\bibitem[DT10]{DT10}
{\scshape S.~Diverio {\normalfont \smfandname} S.~Trapani} -- {\og A remark on
  the codimension of the {G}reen-{G}riffiths locus of generic projective
  hypersurfaces of high degree\fg}, \emph{J. Reine Angew. Math.} \textbf{649}
  (2010), p.~55--61.

\bibitem[Ful98]{Ful13}
{\scshape W.~Fulton} -- \emph{Intersection theory}, second \smfedname,
  Ergebnisse der Mathematik und ihrer Grenzgebiete. 3. Folge. A Series of
  Modern Surveys in Mathematics [Results in Mathematics and Related Areas. 3rd
  Series. A Series of Modern Surveys in Mathematics], vol.~2, Springer-Verlag,
  Berlin, 1998.

\bibitem[GG80]{GG79}
{\scshape M.~Green {\normalfont \smfandname} P.~Griffiths} -- {\og Two
  applications of algebraic geometry to entire holomorphic mappings\fg}, in
  \emph{The {C}hern {S}ymposium 1979 ({P}roc. {I}nternat. {S}ympos.,
  {B}erkeley, {C}alif., 1979)}, Springer, New York-Berlin, 1980, p.~41--74.

\bibitem[Kob70]{Kob70}
{\scshape S.~Kobayashi} -- \emph{Hyperbolic manifolds and holomorphic
  mappings}, Pure and Applied Mathematics, vol.~2, Marcel Dekker, Inc., New
  York, 1970.

\bibitem[Laz04]{Laz04II}
{\scshape R.~Lazarsfeld} -- \emph{Positivity in algebraic geometry.{II}},
  Ergebnisse der Mathematik und ihrer Grenzgebiete. 3. Folge. A Series of
  Modern Surveys in Mathematics [Results in Mathematics and Related Areas. 3rd
  Series. A Series of Modern Surveys in Mathematics], vol.~49, Springer-Verlag,
  Berlin, 2004, Positivity for vector bundles, and multiplier ideals.

\bibitem[{Mer}18]{Mer18}
{\scshape J.~{Merker}} -- {\og {Kobayashi hyperbolicity in degree {$\geqslant
  n^{2n}$}}\fg}, \emph{ArXiv e-prints} (2018).

\bibitem[Nad89]{Nad89}
{\scshape A.~M. Nadel} -- {\og Hyperbolic surfaces in {${\bf P}^3$}\fg},
  \emph{Duke Math. J.} \textbf{58} (1989), no.~3, p.~749--771.

\bibitem[Nak00]{Nak00}
{\scshape M.~Nakamaye} -- {\og Stable base loci of linear series\fg},
  \emph{Mathematische Annalen} \textbf{318} (2000), no.~4, p.~837--847.

\bibitem[Rou10]{Rou10}
{\scshape E.~Rousseau} -- {\og Hyperbolicity of geometric orbifolds\fg},
  \emph{Trans. Amer. Math. Soc.} \textbf{362} (2010), no.~7, p.~3799--3826.

\bibitem[RY18]{RY18}
{\scshape E.~{Riedl} {\normalfont \smfandname} D.~{Yang}} -- {\og {Applications
  of a grassmannian technique in hypersurfaces}\fg}, \emph{ArXiv e-prints}
  (2018).

\bibitem[Siu04]{Siu04}
{\scshape Y.-T. Siu} -- {\og Hyperbolicity in complex geometry\fg}, in
  \emph{The legacy of {N}iels {H}enrik {A}bel}, Springer, Berlin, 2004,
  p.~543--566.

\bibitem[{Xie}16]{Xie16}
{\scshape S.-Y. {Xie}} -- {\og {Generalized Brotbek's symmetric differential
  forms and applications}\fg}, \emph{ArXiv e-prints} (2016).

\bibitem[Xie18]{Xie15}
{\scshape S.-Y. Xie} -- {\og On the ampleness of the cotangent bundles of
  complete intersections\fg}, \emph{Inventiones mathematicae} (2018).

\end{thebibliography}

\end{document}